\documentclass[a4paper,twoside,fleqn]{amsart}
\newtheorem{theorem}{Theorem}[section]
\newtheorem{lemma}[theorem]{Lemma}

\theoremstyle{definition}
\newtheorem{definition}[theorem]{Definition}
\newtheorem{example}[theorem]{Example}

\theoremstyle{remark}
\newtheorem{remark}[theorem]{Remark}

\numberwithin{equation}{section}

\usepackage{url}
\usepackage{amsmath}
\usepackage{amsthm}
\usepackage{todonotes}
\usepackage{amssymb}
\usepackage{enumitem}
\usepackage{todonotes}
\usepackage{soul}
\usepackage[hidelinks]{hyperref}
\newtheorem{corollary}[theorem]{Corollary}
\newtheorem{proposition}[theorem]{Proposition}
\newtheorem{observation}[theorem]{Observation}
\newtheorem{question}[theorem]{Question}
\newtheorem*{claim}{Claim}

\newcommand{\bb}[1]{\mathbb{#1}}
\newcommand{\cal}[1]{\mathcal{#1}}
\renewcommand{\rm}[1]{\mathrm{#1}}

\newcommand{\bA}{\mathbf{A}}

\newcommand{\bB}{\mathbf{B}}

\newcommand{\bC}{\mathbf{C}}

\newcommand{\bK}{\mathbf{K}}

\newcommand{\bP}{\mathbf{P}}

\newcommand{\bQ}{\mathbf{Q}}

\newcommand{\bX}{\mathbf{X}}

\newcommand{\wh}{\widehat}
\newcommand{\ol}{\overline}

\newcommand{\Sa}{\mathrm{Sa}}
\newcommand{\sa}{\mathrm{Sa}}

\newcommand{\fr}{Fra\"iss\'e }
\renewcommand{\phi}{\varphi}
\newcommand{\emb}{\mathrm{Emb}}
\newcommand{\aut}{\mathrm{Aut}}

\newcommand{\im}{\mathrm{Im}}

\newcommand{\rel}{\mathrm{rel}}
\newcommand{\str}{\mathbf{Str}}

\newcommand{\ct}{\mathrm{CT}}

\def\oeis#1{Sequence #1 of OEIS~\cite{oeis}}

\newcommand{\card}[1]{\left| #1 \right|}    

\def\cont{^\frown}
\def\L{\mathrm L}
\def\R{\mathrm R}
\def\X{\mathrm X}
\def\str#1{\mathbf {#1}}
\def\lleq{\mathbin{\leq_{\mathrm{\Sigma}}}}
\def\llt{\mathbin{<_{\mathrm{\Sigma}}}}
\def\lexleq{\mathbin{\leq_{\mathrm{lex}}}}
\def\lexlt{\mathbin{<_{\mathrm{lex}}}}
\def\eltlt{\vartriangleleft}
\def\eltleq{\trianglelefteq}
\def\Alphabet{\Sigma}
\def\rel#1#2{R_{\mathbf{#1}}^{#2}}
\def\Shape#1#2{\mathrm{Shape}\allowbreak(#1,\allowbreak #2)}
\def\nShape#1#2#3{\mathrm{Shape}_{#1}\allowbreak(#2,\allowbreak #3)}
\newcommand{\AmbStr}[1]{(#1, \allowbreak \lexleq\nobreak,\allowbreak\preceq\nobreak,\allowbreak\eltleq)}
\makeatletter
\@namedef{subjclassname@2020}{%
  \textup{2020} Mathematics Subject Classification}
\makeatother

\begin{document}
\title[Big Ramsey degrees of the generic partial order]{Characterisation of the big Ramsey degrees of the generic partial order}
	\author[M. Balko et al.]{Martin Balko}
        \address{Department of Applied Mathematics (KAM), Charles University, Ma\-lo\-stranské~nám\v estí 25, Praha 1, Czech Republic}
	\email{balko@kam.mff.cuni.cz}
	\author[]{David Chodounsk\'y}
        \address{Department of Applied Mathematics (KAM), Charles University, Ma\-lo\-stranské~nám\v estí 25, Praha 1, Czech Republic}
        \address{Institute of Mathematics of the Czech Academy of Sciences, \v{Z}itn\'a~25, Praha~1, Czech Republic}
        \email{chodounsky@math.cas.cz}
	\author[]{Natasha Dobrinen}
	\address{Department of Mathematics, University of Notre Dame,
	255 Hurley Bldg, Notre Dame, IN 46556 USA}
	\email{ndobrine@nd.edu}
	\author[]{Jan Hubi\v cka}
        \address{Department of Applied Mathematics (KAM), Charles University, Ma\-lo\-stranské~nám\v estí 25, Praha 1, Czech Republic}
	\email{hubicka@kam.mff.cuni.cz}
	\author[]{Mat\v ej Kone\v cn\'y}
        \address{Institute of Algebra, TU Dresden, Helmholtzstraße 10, 01069 Dresden, Germany\\and Department of Applied Mathematics (KAM), Charles University, Ma\-lo\-stranské~nám\v estí 25, Praha 1, Czech Republic}
	\email{matej.konecny@tu-dresden.de}
	\author[]{Lluis Vena}
	\address{Universitat Polit\`ecnica de Catalunya, Barcelona, Spain}
	\email{lluis.vena@gmail.com}
	\author[]{Andy Zucker}
	\address{Department of Pure Mathematics, University of Waterloo, 200 University Ave W, Waterloo, ON N2L 3G1, Canada}
	\email{a3zucker@uwaterloo.ca}

\date{}

\dedicatory{}
\begin{abstract}
	As a result of 33 intercontinental Zoom calls, we characterise big
	Ramsey degrees of the generic partial order.
	This is an infinitary extension of the well known fact that
	finite partial orders endowed with linear extensions form a Ramsey class (this result
	was announced by Ne\v set\v ril and R\"odl in 1984).
	Towards this, we refine earlier upper bounds obtained by Hubi\v cka based on a new connection of big Ramsey degrees to the Carlson--Simpson theorem
	and we also introduce a new technique of giving lower bounds using an iterated application of the upper-bound theorem. 
\end{abstract}
\maketitle
\section{Introduction}
\label{sec:introduction}

We characterise the big Ramsey degrees of (finite substructures of) $\str{P}$ in a similar fashion as the authors did in \cite{Balko2021exact} for binary finitely constrained free amalgamation classes. Similarly to that case, our characterisation is in terms of a tree-like object we call a ``poset diary" where each level has exactly one critical event. It follows that any poset diary which codes the generic poset is a big Ramsey structure for $\str{P}$ (Definition~1.3 of \cite{zucker2017}).

Relative to the small Ramsey degree case (see e.g.~\cite{hubicka2025twenty,NVT14,Bodirsky2015,hubicka2020structural}), 
there are fewer classes of structures for which
big Ramsey degrees are fully understood.  The following is the current state of the art (see~\cite{hubicka2024survey,hubicka2025twenty,dobrinen2021ramsey}):

 \begin{enumerate}
	 \item The Ramsey theorem implies that the big Ramsey degree of every finite linear order in the linear order $\omega$ is 1.
	 \item In 1979, Devlin refined upper bounds by Laver and characterised big Ramsey degrees of the rational linear order~\cite{devlin1979,todorcevic2010introduction}, see also \oeis{A000182}.
	 \item In 2006 Laflamme, Sauer and Vuksanovi\'c characterised big Ramsey degrees of the Rado (or random) graph and related random structures in binary languages~\cite{Laflamme2006}. Actual numbers were counted by Larson~\cite{larson2008counting}, see also \oeis{A293158}.
	\item In 
 2008
 Nguyen Van Thé characterised big Ramsey degrees of the ultrametric {U}rysohn spaces~\cite{NVT2008}.
	\item In 2010 Laflamme, Nguyen Van Thé and Sauer~\cite{laflamme2010partition} characterised the big Ramsey degrees of the dense local order, $\mathbf{S}(2)$
 and the $\mathbb{Q}_n$ structures.
\item		 
  In 2020
		 Coulson, Dobrinen, and Patel  in ~\cite{coulson2022indivisibility} and ~\cite{coulson2022SDAP} 
   characterised  the big Ramsey degrees of
homogeneous binary relational structures  which satisfy  SDAP$^+$,
recovering work  
in~\cite{Laflamme2006} 
and~\cite{laflamme2010partition} and characterising the big Ramsey degrees of the ordered versions of structures in~\cite{Laflamme2006},
the generic $n$-partite and generic ordered $n$-partite graphs, and the $(\mathbb{Q}_{\mathbb{Q}})_n$ hierarchy.
	\item In 2020 Barbosa characterised the big Ramsey degrees of  the directed circular orders $\mathbf{S}(k)$, $k\geq 2$, \cite{barbosa2020categorical}.
	\item  In 2020 a characterisation of the big Ramsey degrees of the triangle-free Henson graph was obtained by Dobrinen~\cite{dobrinen2020ramsey} and independently by the remaining authors of this article. Joining efforts,
		the authors were able to fully characterise big Ramsey degrees of free amalgamation classes in finite binary languages described by finitely many forbidden cliques~\cite{Balko2021exact}. Exact counts for the case of $K_3$ and $K_4$-free graphs were recently determined by Hubička, Konečný, Vodseďálek, and Zucker~\cite{Vodsedalek2025,Vodsedalek2025bc}, see also \oeis{A387346}.
 \end{enumerate}

If one draws an analogy to the small Ramsey results, many of the  aforementioned
characterisations can be seen as infinitary generalisations of special cases of
the Ne\v set\v ril--R\"odl theorem~\cite{Nevsetvril1977}. Small Ramsey degrees (or the \emph{Ramsey expansions} satisfying the \emph{expansion property}, see~\cite{NVT14})
are known for many other classes including the class of partial orders with linear extensions~\cite{Nevsetvril1984,Trotter1985,sokic2012ramsey,masulovic2016pre,solecki2017ramsey,nevsetvril2018ramsey}, metric spaces~\cite{Nevsetvril2007,Dellamonica2012,masulovic2016pre} and other examples satisfying rather general structural conditions~\cite{Hubicka2016} (see also recent surveys \cite{hubicka2025twenty,hubicka2020structural,Konecny2023phd}).

In this paper, we characterize the big Ramsey degrees of the
generic partial order. This class represents an important new example since its finitary
counterpart is not a consequence of the Ne\v set\v ril--R\"odl theorem~\cite{Nevsetvril1977}, nor is it of the form covered by  \cite{coulson2022SDAP}.
To complete our characterisation, we further refine the proof techniques introduced in \cite{Hubicka2020CS} used to prove finite upper bounds, which were based on the Carlson--Simpson
theorem, and we also give a fundamentally different approach than in \cite{Balko2021exact} to prove the lower bounds. The techniques introduced in this paper
generalize to other classes as we briefly outline in the final section. However, to keep the paper simple, we did not attempt to state the results
in the greatest possible generality.

Given structures $\str{A}$ and $\str{B}$, we denote by $\emb(\bA, \bB)$ the set
of all embeddings from $\str{A}$ to $\str{B}$. We write $\str{C}\longrightarrow (\str{B})^\str{A}_{k,\ell}$ to denote the following statement:
\begin{quote}
	For every colouring $\chi\colon \emb(\bA, \bC)\to k$, there exists
	$f\in \emb(\bB, \bC)$ such that $\chi$  takes no more than $l$ values on $f\circ \emb(\bA, \bB)$.
\end{quote}
For a countably infinite structure $\str{B}$ and a finite substructure $\str{A}$, the \emph{big Ramsey degree} of $\str{A}$ in $\str{B}$ is
the least number $\ell\in \mathbb N\cup \{\infty\}$ such that $\str{B}\longrightarrow (\str{B})^\str{A}_{k,\ell}$ for every $k\in \mathbb N$.
Similarly, if $\mathcal K$ is a class of finite structures and $\str{A}\in\mathcal K$, the \emph{small Ramsey degree} of $\str{A}$ in~$\mathcal K$ is the least number $\ell\in \mathbb N\cup \{\infty\}$ such that for every $\str{B}\in \mathcal K$ and $k\in \mathbb N$ there exists $\str{C}\in \mathcal K$ such that $\str{C}\longrightarrow (\str{B})^\str{A}_{k,l}$.

A structure is \emph{homogeneous} if every isomorphism between
two of its finite
substructures extends to an automorphism.
It is well known that up to
isomorphism there is a unique homogeneous partial order $\str{P}=(P,\leq_\str{P})$ such that
every countable partial order has an embedding into $\str{P}$. The order $\str{P}$ is called
the \emph{generic partial order} (see e.g.~\cite{Macpherson2011}). We refine the following recent result.

\begin{theorem}[Hubi\v cka~\cite{Hubicka2020CS}]
	The big Ramsey degree of every finite partial order
	in the
	generic partial order $\str{P}$ is finite.
\end{theorem}

Our construction makes use of the following partial order introduced in~\cite{Hubicka2020CS}; this is closely related to
the tree of 1-types within a fixed enumeration of $\str{P}$ discussed in Section~\ref{sec:1types}.
Let $$\Sigma=\{\L,\X,\R\}$$ be an \emph{alphabet} ordered by $\llt$ as
$$\L\llt \X\llt \R.$$  We denote by
$\Sigma^*$ the set of all finite words in the alphabet $\Sigma$, by
$\lexleq$ their lexicographic order, and by $|w|$ the length of the word $w = w_0w_1\cdots w_{|w|-1}$. We denote the empty word by $\emptyset$. Given words $w,w'\in \Sigma^*$, we write $w\sqsubseteq w'$ if $w$ is an initial segment of $w'$.
\begin{definition}[Partial order $(\Sigma^*,\preceq)$]
	For $w,w'\in \Sigma^*$, we set $w\prec w'$ if and only if there exists $i$ such that:
	\begin{enumerate}
		\item  $0\leq i<\min(|w|,|w'|)$,
		\item $(w_i,w'_i)=(\L,\R)$, and
		\item $w_j\lleq w'_j$ for every $0\leq j< i$.
	\end{enumerate}
\end{definition}

We call the least $i$ satisfying the condition above the \emph{witness} of $w\prec w'$ and denote it by $i(w,w')$. We write $w\preceq w'$ if either $w\prec w'$ or $w=w'$.
\begin{proposition}[\cite{Hubicka2020CS}]
	\label{prop:pos}
	$(\Sigma^*,\preceq)$ is a partial order, and  $(\Sigma^*,\lexleq)$ is a linear extension of $(\Sigma^*,\preceq)$.
\end{proposition}
This partial order will serve as the main tool for characterising the big Ramsey degrees of $\str{P}$.
The intuition for its definition is described in Section~\ref{sec:1types}.
\begin{proof}[Proof of Proposition~\ref{prop:pos}]
	It is easy to see that $\preceq$ is reflexive and anti-symmetric.  We verify transitivity. Let $w\prec w'\prec w''$
	and put $i=\min(i(w,w'),i(w',w''))$.

	First assume that $i=i(w,w')$.  Then we have $w_i=\L$ and $w'_i=\R$, which implies that $w''_i=\R$.
	For every $0\leq j<i$, we have $w_j\lleq w'_j\lleq w''_j$.
	The transitivity of $\lleq$ then implies that $w\prec w''$, so $i(w,w'')$ exists with $i(w,w'')\leq i$.

	Now assume that $i=i(w',w'')$. Then we have $w'_i=\L$ and $w''_i=\R$, and since $w'_i=\L$, we also have $w_i=\L$.
	Again, for every $0\leq j<i$ it holds that $w_j\lleq w'_j\lleq w''_j$.
	Similarly as above, we have $w\preceq w''$ with $i(w,w'')\leq i$.
\end{proof}

Given a word $w$ and an integer $i \geq 0$, we denote by $w|_i$ the \emph{initial segment} of $w$ of length $i$.
For $S\subseteq \Sigma^*$, we let $\overline{S}$ be the set $\{w|_i\colon w\in S, 0\leq i\leq |w|\}$.
Given $\ell\geq 0$, we let $S_\ell=\{w\in S:|w|=\ell\}$ and call it the \emph{level $\ell$ of $S$}. When writing $\overline{S}_\ell$, we always mean level $\ell$ of $\overline{S}$. 
A word $w\in S$ is called a \emph{leaf} of $S$ if there is no word $w'\in S$ extending $w$.
Given a word $w$ and a letter $c\in \Sigma$, we denote by $w\cont c$
the word created by adding $c$ to the end of $w$. We also set $S\cont c=\{w\cont c:w\in S\}$. Given $\ell\geq 0$, $\Sigma^*_\ell = (\Sigma^*)_\ell$ denotes the set of words of length $\ell$ in the alphabet $\Sigma$.
\begin{definition}[Partial orders $(\Sigma^*_\ell,\eltleq)$]
	Given $\ell\geq 0$ and words $w, w'\in \Sigma^*_\ell$, we write $w\eltleq w'$ if 
	$w_i\lleq w'_i$ for every $0\leq i<\ell$. We write $w\perp w'$ if $w$ and $w'$ are $\eltleq$-incomparable (that is, if neither of $w\eltleq w'$ nor $w'\eltleq w$ holds). We call $w$ and $w'$ \emph{related} if one of expressions $w\preceq w'$, $w'\preceq w$ or $w\perp w'$ holds, otherwise they are \emph{unrelated}.
\end{definition}

Intuitively, $s\eltleq t$ describes those pairs of nodes on the same level where it is possible to extend $s$ and $t$ to nodes with $s'\prec t'$. However, observe that (somewhat counter-intuitively)
it can happen that both $\preceq$ and $\perp$  hold at the same time: For example, we have both that $\L\R\preceq \R\L$ as well as $\L\R\perp \R\L$. Later, we will introduce the notion of compatibility to help us handle these situations.
While it is easy to check that $(\Sigma_\ell^*,\eltleq)$ is a partial order for every $\ell\geq 0$, we do not extend it to all of $\Sigma^*$. 

To characterise the big Ramsey degrees of $\str{P}$, we introduce the following  technical definition. Our main theorem characterizes the big Ramsey degree of a given finite poset in $\str P$ as the number of \emph{poset-diaries} of a certain kind. The intuition behind the definition will be explained in Section~\ref{sec:1types}.

\begin{definition}[Poset-diaries]
	\label{def:posetdiary}
	A set $S\subseteq \Sigma^*$ is called a \emph{poset-diary} if no member of $S$ extends any other (i.e.,\ $S$ is an antichain in $(\Sigma^*, \sqsubseteq)$) and precisely one of the following four conditions is satisfied for every level $\ell$ with $0\leq \ell< \sup_{w\in S}|w|$:
	\begin{enumerate}
		\item \textbf{Leaf:}  There is $w\in \overline{S}_\ell$ related to every $u\in \overline{S}_\ell\setminus\{w\}$ and
		      \begin{align*}
			      \qquad \overline{S}_{\ell+1} & =(\overline{S}_\ell\setminus \{w\} )\cont \X.
		      \end{align*}
		\item \textbf{Splitting:}  There is $w\in \overline{S}_\ell$ such that
		      \begin{align*}
			      \begin{split}
				      \qquad \overline{S}_{\ell+1}&=\{z\in \overline{S}_\ell:z\lexlt w\}\cont \X\\
				      &\qquad \cup\{w\cont \X,w\cont\R\}\\
				      &\qquad \cup \{z\in \overline{S}_\ell:w\lexlt z\}\cont \R.
			      \end{split}
		      \end{align*}
		\item \textbf{New $\perp$:} There are unrelated words $v\lexlt w\in \overline{S}_\ell$ such that the following is satisfied.
		      \begin{enumerate}[label=(\Alph*)]
			      \item\label{A2} For every $u\in \overline{S}_\ell$ with $v\lexlt u\lexlt w$, we have $u\perp v$ or $u\perp w$. 
		      \end{enumerate}
			Moreover:
		      \begin{align*}
			      \begin{split}
				      \qquad \overline{S}_{\ell+1}&=\{z\in \overline{S}_\ell:z\lexlt v\}\cont \X\\
				      &\qquad \cup \{v\cont \R\}\\
				      &\qquad \cup\{z\in \overline{S}_\ell:v\lexlt z\lexlt w\hbox{ and }z\perp v\}\cont \X\\
				      &\qquad \cup\{z\in \overline{S}_\ell:v\lexlt z\lexlt w\hbox{ and } z\not\perp v\}\cont \R\\
				      &\qquad \cup \{{w}\cont \X\}\\
				      &\qquad \cup \{z\in \overline{S}_\ell:w\lexlt z\}\cont \R.
			      \end{split}
		      \end{align*}
		\item \textbf{New $\prec$:}  There are unrelated words $v\lexlt w\in \overline{S}_\ell$ such that the following is satisfied.
		      \begin{enumerate}[label=(B\arabic*)]
			      \item\label{B1} For every $u\in \overline{S}_\ell$ with $u\lexlt v$, we have $u\preceq w$ or $u\perp v$.
			      \item\label{B2} For every $u\in \overline{S}_\ell$ with $w\lexlt u$, we have $v\preceq u$ or $w\perp u$.
		      \end{enumerate}
			Moreover:
		      \begin{align*}
			      \begin{split}
				      \qquad \overline{S}_{\ell+1}&=\{z\in \overline{S}_\ell:z\lexlt v\hbox{ and }z\perp v\}\cont \X \\
				      &\qquad \cup \{z\in \overline{S}_\ell:z\lexlt v\hbox{ and }z\mathbin{\not \perp} v\}\cont \L\\
				      &\qquad \cup \{v\cont \L\}\\
				      &\qquad \cup\{z\in \overline{S}_\ell:v\lexlt z\lexlt w\}\cont \X\\
				      &\qquad \cup \{{w}\cont \R\}\\
				      &\qquad \cup \{z\in \overline{S}_\ell:w\lexlt z\hbox{ and }w\perp z\}\cont\X\\
				      &\qquad \cup \{z\in \overline{S}_\ell:w\lexlt z\hbox{ and }w\not\perp z\}\cont\R.
			      \end{split}
		      \end{align*}
	\end{enumerate}
	\begin{figure}
		\includegraphics{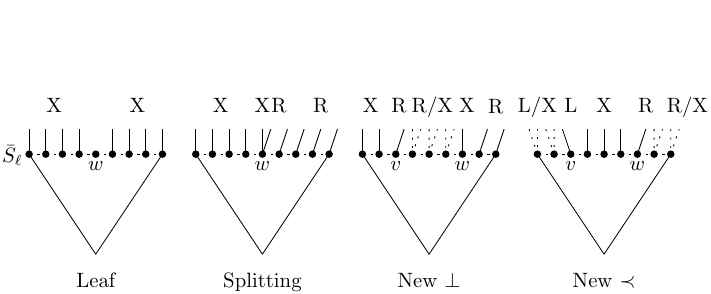}
		\caption{Possible levels in poset-diaries.}
		\label{fig:levels}
	\end{figure}
	See also Figure~\ref{fig:levels}.
\end{definition}

Given a countable partial order $\str{Q}$, we let $T(\str{Q})$ be the set of all poset-diaries $S$ such that $(S,\preceq)$ is isomorphic to $\str{Q}$.
\begin{example}
	Denote by $\str{A}_n$ the anti-chain  with $n$ vertices and by $\str{C}_n$
	the chain with $n$ vertices.
	\begin{align*}
		T(\str{A}_1)=T(\str{C}_1) & =\{\emptyset\},                       \\
		T(\str{A}_2)              & =\{\{\X\R,\R\X\X\},\{\X\R\X,\R\X\}\}, \\
		T(\str{C}_2)              & =\{\{\X\L,\R\R\X\},\{\X\L\X,\R\R\}\}.
	\end{align*}
	\begin{figure}
		\includegraphics[page=1]{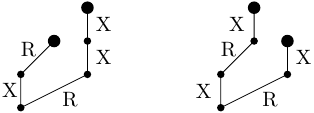}
		\caption{Diaries of $\str{A}_2$.}
		\label{fig:types1}
	\end{figure}
	\begin{figure}
		\includegraphics[page=2]{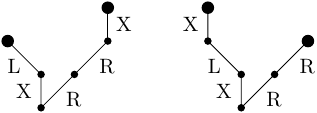}
		\caption{Poset-diaries of $\str{C}_2$.}
		\label{fig:types2}
	\end{figure}
	In all diaries in $T(\str{A}_2)\cup T(\str{C}_2)$, level 0 is splitting, level 1 adds a new $\perp$ or $\prec$, and levels 2 and 3 are leaves.

	Poset-diaries of small partial orders can be determined by an exhaustive search
	tool.\footnote{\url{https://github.com/janhubicka/big-ramsey}}
		 We determined that $|T(\str{C}_3)|=52$, $|T(\str{C}_4)|=11000$,
	$|T(\str{A}_3)|=84$, and $|T(\str{A}_4)|=75672$.  Overall there are:
	\begin{enumerate}
		\item 1 poset-diary of the (unique) partial order of size 1: $S=\{\emptyset\}$,
		\item 4 poset-diaries of partial orders of size 2: $T(\str{A}_2)\cup T(\str{C}_2)$,
		\item 464 poset-diaries of partial orders of size 3,
		\item 1874880 poset-diaries of partial orders of size 4.
	\end{enumerate}
\end{example}
As our main result, we determine the big Ramsey degrees of $\str{P}$ and show that $\str{P}$ admits a big Ramsey structure; while we refer to \cite{zucker2017} for the precise definition (see also \cite{hubicka2024survey}), a big Ramsey structure for $\str{P}$ is an expansion $\str{P}^*$ of $\str{P}$ which encodes the exact big Ramsey degrees for all finite substructures simultaneously in a coherent fashion. 

\begin{theorem}
\label{thm:main}
	For every finite partial order $\str{Q}$, the big Ramsey degree of $\str{Q}$ in the generic partial order $\str{P}$ equals $|T(\str{Q})|\cdot |\mathrm{Aut}(\str{Q})|$. Furthermore, any $\str{P}^*\in T(\str{P})$ encodes a big Ramsey structure for $\str{P}$. Consequently, the topological group $\mathrm{Aut}(\str{P})$ admits a metrizable universal completion flow.
\end{theorem}
Note that the number of poset-diaries is multiplied by the size of the automorphism group since we
define big Ramsey degrees with respect to embeddings (as done, for example, in~\cite{zucker2017}).  
 Big Ramsey degrees are often
defined with respect to substructures (see, for example, \cite{devlin1979,Laflamme2006,larson2008counting,NVT2008,laflamme2010partition}) and in that case the degree would be $|T(\str{Q})|$.
\begin{example}
	\begin{align*}
		|T(\str{A}_1)|\cdot |\mathrm{Aut}(\str{A}_1)|=|T(\str{C}_1)|\cdot |\mathrm{Aut}(\str{C}_1)| & =1,       \\
		|T(\str{A}_2)|\cdot |\mathrm{Aut}(\str{A}_2)|                                               & =4,       \\
		|T(\str{C}_2)|\cdot |\mathrm{Aut}(\str{C}_2)|                                               & =2,       \\
		|T(\str{A}_3)|\cdot |\mathrm{Aut}(\str{A}_3)|                                               & =504,     \\
		|T(\str{C}_3)|\cdot |\mathrm{Aut}(\str{C}_3)|                                               & =52,      \\
		|T(\str{A}_4)|\cdot |\mathrm{Aut}(\str{A}_4)|                                               & =1816128, \\
		|T(\str{C}_4)|\cdot |\mathrm{Aut}(\str{C}_4)|                                               & =11000.   \\
	\end{align*}
\end{example}

\section{Preliminaries}\label{sec:preliminaries}
\subsection{Relational structures}
We use the standard model-theoretic notion of a relational structure.
Let $L$ be a language with relational symbols $\rel{}{}\in L$ each equipped with a positive natural number called its {\em arity}.
An \emph{$L$-structure} $\str{A}$ on $A$ is a structure with {\em vertex set} $A$ and  relations $\rel{A}{}\subseteq A^r$ for every symbol $\rel{}{}\in L$ of arity $r$.  We typically use bold letters for structures and the corresponding unbolded letters for the underlying set unless otherwise specified. If the set $A$ is finite, countable, countably infinite, etc., we call $\str A$ a \emph{finite}, \emph{countable}, or \emph{countably infinite structure}, respectively. All structures that we consider are countable.

Since we work with structures in multiple languages, we will list the vertex
set along with the relations of the structure, e.g., $(P,\leq)$ for partial orders.
\subsection{Trees}
For us, a \emph{tree} is a (possibly empty) partially ordered set $(T, <_T)$ such
that, for every $t \in T$, the set $\{s \in T : s <_T t \}$ is finite and linearly ordered by $<_T$.
All nonempty trees we consider are \emph{rooted}, that is, they have a unique minimal element called the \emph{root} of the tree.
An element $t\in T$ of a tree $T$ is called a \emph{node} of $T$ and its \emph{level},
denoted by $\card{t}_T$, is the size of the set $\{s \in T : s <_T t\}$.
Note that the root has level~0.
Given a tree $T$ and nodes $s, t \in T$, we say that $s$ is a \emph{successor} of $t$ in $T$ if $t \mathbin{\leq_T} s$.
A \emph{subtree} of a tree $T$ is a subset $T'$ of $T$ viewed as a tree
equipped with the induced partial ordering.

Given words $w,w'\in \Sigma^*$, we write $w\sqsubseteq w'$ if $w$ is an initial segment of $w'$. With this partial order we obtain the tree $(\Sigma^*,\sqsubseteq)$
and the notation on words introduced in Section~\ref{sec:introduction} can be viewed as a special case
of the notation introduced here.
\subsection{Parameter words}
To obtain upper bounds for the big Ramsey degrees of $\str{P}$, we apply a Ramsey theorem for parameter words which we briefly review now.

Given a finite alphabet $\Alphabet$ and $k\in \omega\cup \{\omega\}$, a \emph{$k$-parameter word} is a (possibly infinite) string $W$ in the
alphabet $\Alphabet\cup \{\lambda_i\colon 0\leq i<k\}$ containing all symbols $\lambda_i$, $0\leq i < k$ such that, for every $1\leq j < k$, the first
occurrence of $\lambda_j$ appears after the first occurrence of $\lambda_{j-1}$.
The symbols $\lambda_i$ are called \emph{parameters}.
We will use uppercase letters to denote parameter words and lowercase letters for  words
without parameters.
Let $W$ be an $n$-parameter word and let $U$ be a parameter word of length $k\leq n$, where $k,n\in \omega\cup\{\omega\}$. Then
$W(U)$ is the parameter word created by \emph{substituting} $U$ to $W$. More precisely, $W(U)$ is created from~$W$ by replacing each occurrence of $\lambda_i$, $0\leq i < k$, by $U_i$ and truncating it just
before the first occurrence of $\lambda_k$ in $W$.

We apply the following Ramsey theorem for parameter words, which is  an easy consequence of the Carlson--Simpson theorem\cite{carlson1984}, see also \cite{todorcevic2010introduction,Karagiannis2013}:
\begin{theorem}
	\label{thm:CS}
	Let $\Alphabet$ be a finite alphabet.
	If $\Alphabet^*$ is coloured with finitely many colours, then there exists an infinite-parameter word
	$W$ such that $$W[\Alphabet^*]=\{W(s)\colon s\in \Alphabet^*\}$$ is monochromatic.
\end{theorem}

\section{Tree of 1-types}\label{sec:1types}

Poset-diaries, which can be compared to \emph{Devlin
	embedding types} (see Chapter 6.3 of \cite{todorcevic2010introduction}) or \emph{diaries} introduced in \cite{Balko2021exact}, have an intuitive meaning when
understood in the context of the tree of 1-types of $\str{P}$. We now introduce this tree and its enrichment to an \emph{aged coding tree} and discuss how poset-diaries can be obtained as a suitable abstraction of the aged coding tree. 

An \emph{enumerated structure} is simply a structure $\str{A}$ with underlying set $A = |A|$. Fix  a countably infinite enumerated structure $\str{A}$.  Given vertices $u,v$ and an integer $n$ satisfying $\min(u,v)\geq n\geq 0$, we write $u\sim^\str{A}_n v$ and say that \emph{$u$ and $v$ have the same (quantifier-free) type over $\{0,1,\ldots,n-1\}$}, if the structures induced by $\str{A}$ on $\{0,1,\ldots, n-1,u\}$ and $\{0,1,\ldots, n-1,v\}$ are isomorphic via the map which is the identity on $\{0,...,n-1\}$ and sends $u$ to $v$. We write $[u]^\str{A}_n$ for the $\sim^\str{A}_n$-equivalence class of the vertex $u$.
\begin{definition}[Tree of 1-types]
	Let $\str{A}$ be a countably infinite (relational) enumerated structure. Given $n< \omega$, write $\bb{T}_\str{A}(n) = \omega/\!\sim^{\str{A}}_n$. A (quantifier-free) \emph{1-type} is any member of the disjoint union $\bb{T}_\str{A}:=\bigsqcup_{n<\omega} \bb{T}_\str{A}(n)$. We turn $\bb{T}_\str{A}$ into a tree as follows. Given $x\in \bb{T}_\str{A}(m)$ and $y\in \bb{T}_\str{A}(n)$, we declare
	that $x\leq^{\mathbb T}_{\str{A}} y$ if and only if $m\leq n$ and $x\supseteq y$. 
        We will denote by $[n]^\str{A}_n$ the equivalence class of $n$ under $\sim^{\str{A}}_n$.
	
	In the case that we have a \fr class $\cal{K}$ in mind (which for us will always be the class of finite partial orders), we can extend the definition to a finite enumerated $\str{A}\in \cal{K}$ as follows. Fix an enumerated \fr limit $\str{K}$ of $\cal{K}$ which has $\str{A}$ as an initial segment. We then set $\bb{T}_\str{A} = \bb{T}_\str{K}({<}|\str{A}|)$. This does not depend on the choice of $\str{K}$.
\end{definition}

In the case that $\str{A}$ is a structure in a finite binary relational language, we can encode $\bb{T}_{\str{A}}$ as a subtree of $k^{<\omega}$ for some $k< \omega$ as follows. Given two enumerated structures $\str{B}$ and $\str{C}$, an \emph{ordered embedding} of $\str{B}$ into $\str{C}$ is any embedding of $\str{B}$ into $\str{C}$ which is an increasing injection of the underlying sets. Write $\mathrm{OEmb}(\str{B}, \str{C})$ for the set of ordered embeddings of $\str{B}$ into $\str{C}$. Fix once and for all an enumeration $\{\str{B}_i: i< k\}$ of the set of enumerated structures of size $2$ which admit an enumerated embedding into $\str{A}$. Given $x\in \bb{T}_\str{A}(m)$, we define $\sigma(x)\in k^m$ (a word in alphabet $\{0,1,\ldots,k-1\}$ of length $m$) where given $j< m$, we set $\sigma(x)_j = i$ (recall that we use subscripts to index letters in a word) if and only if for some (equivalently every) $n\in x$ there is $f\in \mathrm{OEmb}(\str{B}_i, \str{A})$ with $\mathrm{Im}(f) = \{j, n\}$. The map $\sigma\colon \bb{T}_\str{A}\to k^{< \omega}$ is then an embedding of trees. We write $\mathrm{CT}^\str{A} = \sigma[\bb{T}_\str{A}]$ and call this the \emph{coding tree} of $\str{A}$. Typically we also endow $\mathrm{CT}^\str{A}$ with \emph{coding nodes}, where for each $n$, the $n^{\mathrm{th}}$ coding node of $\mathrm{CT}^\str{A}$ is defined to be $c^\str{A}(n) := \sigma([n]^\str{A}_n)$.

The tree of 1-types of a given structure is useful for constructing  unavoidable colourings as well as for proving upper bounds on big Ramsey degrees; see for instance 
\cite{Laflamme2006,
dobrinen2017universal,
dobrinen2019ramsey,
coulson2022indivisibility,coulson2022SDAP,zucker2020,Hubicka2020CS,Hubickabigramsey,Hubicka2020uniform,Balko2021exact,dobrinen2020exposition,dobrinen2020forcing,hubicka2024survey,hubicka2025twenty}.
We therefore fix an enumerated generic partial order $\str{P}$, and we put $(\mathbb T,\leq_\mathbb T\nobreak )=(\mathbb T_\str{P},\leq^\mathbb T_\str{P})$ and $\mathrm{CT} = \mathrm{CT}^\str{P}$. Noting that $\bP$ is a binary relational structure and that there are exactly $3$ enumerated posets of size $2$, we have $k = 3$; by identifying the symbols $\{\L, \X, \R\}$ with $\{0, 1, 2\}$, we can identify $\mathrm{CT}$ as a subtree of $(\Sigma^*, \sqsubseteq)$. More concretely, given $x\in \bb{T}(m)$ and $j< m$, we have:
$$\sigma(x)_j=\begin{cases}\L & \hbox{if $a\mathbin{<_\str{P}}j$ for every (some) $a\in x$,} \\
	\R & \hbox{if $j\mathbin{<_\str{P}}a$ for every (some) $a\in x$,} \\
             \X & \hbox{otherwise.}
	\end{cases}
$$
	\begin{figure}
		\includegraphics{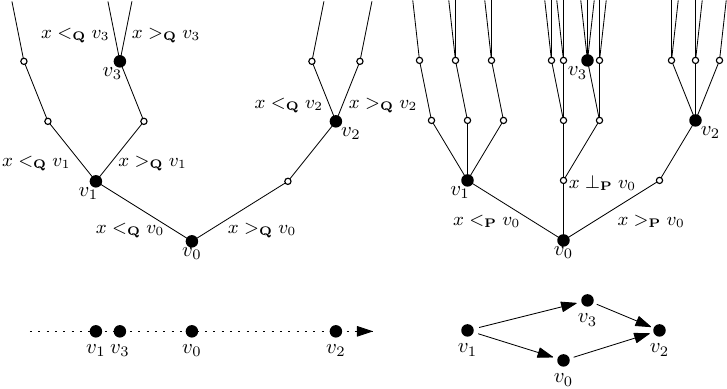}
		\caption{Initial part of the tree of types of an enumerated linear order $(\str{Q},\leq_\str{Q})$ (left) and of the enumerated partial order $\str{P}$ (right). The bold node on each level corresponds to the coding node.}
		\label{fig:types}
	\end{figure}

We remark that $\ct$ is a proper subset of $\Sigma^*$, no matter the enumeration we choose.  For example, if $\L\R\in \ct$, then $\R\L\notin \ct$, since that would imply the existence of vertices $a,b\in \str{P}$ such that $a\mathbin{<_\str{P}} 0\mathbin{<_\str{P}}b$ and $b<\mathbin{_\str{P}} 1\mathbin{<_\str{P}}a$, contradicting the fact that $\str{P}$ is a partial order.

The relations $\prec$, $\lexlt$ and $\perp$, introduced in Section~\ref{sec:introduction}, capture the following properties of types:
\begin{proposition}
	\label{prop:1typerel}
	Let $x\in \bb{T}(m)$ and $y\in \bb{T}(n)$ be 1-types of $\str{P}$.
	\begin{enumerate}[label=(\arabic*)]
		\item\label{prop:1typerel:p0} If there exist $a\in x$ and $b\in y$ satisfying $a\mathbin{<_\str{P}} b$, then for every $\ell<\min(m,n)$ it holds that $\sigma(x)_\ell\lleq \sigma(y)_\ell$.
		\item\label{prop:1typerel:p1} If $\sigma(x)\prec \sigma(y)$, then for every $a\in x$ and $b\in y$ it holds that $a\mathbin{<_\str{P}} b$.
		\item\label{prop:1typerel:p2} If $\sigma(x)\lexlt \sigma(y)$, then for every $a\in x$ and $b\in y$ it holds that $b\mathbin{\not <_\str{P}} a$.
		\item\label{prop:1typerel:p3} If $m = n$ and $\sigma(x)\perp \sigma(y)$, then for every $a\in x$ and $b\in y$ it holds that $a$ and $b$ are $\leq_\str{P}$ incomparable.
	\end{enumerate}
\end{proposition}
\begin{proof}
	We first verify~\ref{prop:1typerel:p0} by contrapositive. Assume there is $\ell<\min(m,n)$ such that $\sigma(y)_\ell \llt \sigma(x)_\ell$.
	First consider the case that $(\sigma(x)_\ell,\sigma(y)_\ell)=(\X,\L)$. It follows for any $a\in x$ and $b\in y$ that $a$ is $\leq_\str{P}$-incomparable with $\ell$ and $b\mathbin{<_\str{P}} \ell$. It follows that we cannot have $a\mathbin{<_\str{P}} b$. The arguments in the cases $(\sigma(x)_\ell,\sigma(y)_\ell)=(\R,\L)$ and $(\sigma(x)_\ell,\sigma(y)_\ell)=(\R,\X)$ are similar.

	To see \ref{prop:1typerel:p1}, observe that $\sigma(x)\prec \sigma(y)$ implies the existence of a vertex $\ell\in \str{P}$ satisfying $\ell<\min(m,n)$ and $(\sigma(x)_\ell,\sigma(y)_\ell)=(\L,\R)$. It follows that for any $a\in x$ and $b\in y$ we have $a\mathbin{<_\str{P}} \ell\mathbin{<_\str{P}} b$.

	To verify~\ref{prop:1typerel:p2}, observe that there exists $\ell<\min(m, n)$ such that $\sigma(x)_\ell\llt \sigma(y)_\ell$. Thus we cannot have $a\in x$, $b\in y$ such that $b\mathbin{<_\str{P}} a$, as this would contradict~\ref{prop:1typerel:p0}.

	Finally to verify~\ref{prop:1typerel:p3}, observe that $\sigma(x)\perp \sigma(y)$ implies the existence of vertices $k, \ell\in \str{P}$ satisfying $\max(k, \ell)<\min(m,n)$, $\sigma(x)_k\llt \sigma(y)_k$ and $\sigma(y)_\ell\llt \sigma(x)_\ell$.
	Hence the existence of $a\in x$ and $b\in y$ with either $a\mathbin{<_\str{P}} b$ or $b\mathbin{<_\str{P}} a$ contradicts~\ref{prop:1typerel:p0}.
\end{proof}

The main difficulty in working with the tree $(\mathbb T,\leq_\mathbb T)$ is the fact that it depends on the choice of an enumeration of $\str{P}$.
For this reason we will focus on the tree $(\Sigma^*,\sqsubseteq)$ which can be seen as an amalgamation of all possible trees $(\mathbb T,\leq_\mathbb T)$ constructed
using all possible enumerations of $\str{P}$.
The next definition captures the main properties of words in $\ct$ which are independent of the choice of enumeration of $\str{P}$.
\begin{definition}[Compatibility]
	\label{defn:comp}
	Words $u\lexleq v\in \Sigma^*$ are \emph{compatible} if the following two conditions are
	satisfied:
	\begin{enumerate}[label=(\arabic*)]
		\item\label{defn:comp:p1} there is no $\ell<\min(|u|,|v|)$ such that $(u_\ell,v_\ell)=(\R,\L)$, and
		\item\label{defn:comp:p2} if there exists $\ell'<\min(|u|,|v|)$ such that $(u_{\ell'},v_{\ell'})=(\L,\R)$, then for every $\ell''<\min(|u|,|v|)$ it holds that $u_{\ell''}\lleq v_{\ell''}$.
	\end{enumerate}
\end{definition}
Intuitively, the words $u$ and $v$ are compatible if and only if they can appear in a coding tree of some enumeration of $\str{P}$.  If there exists $\ell<\min(|u|,|v|)$ such that $(u_\ell,v_\ell)=(\R,\L)$, then in this enumeration we must have $|u|\prec \ell\prec |v|$.  Failures of condition~\ref{defn:comp:p2} of Definition~\ref{defn:comp} contradict transitivity of $\str{P}$.
\begin{proposition}
	Every $s, t\in \ct$ are compatible.
\end{proposition}
\begin{proof}
	Suppose $x, y\in \bb{T}$ are such that $\sigma(x) \lexlt \sigma(y)$.
	To see property \ref{defn:comp:p1} of Definition~\ref{defn:comp},
	suppose there were $\ell<\min(i,j)$ such that $(\sigma(x)_\ell,\sigma(y)_\ell)=(\R,\L)$. This implies that $b\mathbin{<_\str{P}} \ell \mathbin{\leq_\str{P}} a$ for every $a\in x$ and $b\in y$
	which contradicts Proposition~\ref{prop:1typerel}~\ref{prop:1typerel:p2}.

	Property \ref{defn:comp:p2} of Definition~\ref{defn:comp} is a consequence of Proposition~\ref{prop:1typerel}~\ref{prop:1typerel:p0} and the fact that the existence of
	$\ell'$ such that $(\sigma(x)_{\ell'},\sigma(y)_{\ell'})=(\L,\R)$ implies that $a\mathbin{<_\str{P}} \ell\mathbin{<_\str{P}} b$ for every $a\in x$ and $b\in y$.
\end{proof}
In \cite{zucker2020}, using ideas implicit in the parallel $1$'s of \cite{dobrinen2017universal,sauer1998} and pre-$a$-cliques of \cite{dobrinen2019ramsey}, levels of coding trees are endowed with the structure of \emph{aged sets}. This means that for every $m< \omega$, every set $S\subseteq \ct(m)$ is equipped with a class of finite $|S|$-labeled structures describing exactly which finite structures can be coded by coding nodes above the members of $S$. For the generic partial order, it will be useful to encode this information slightly differently than in \cite{zucker2020,Balko2021exact}, in particular since we want to do this on all of $\Sigma^*$, not just on $\ct$. 
\begin{definition}[Level structure]
	Given $\ell\geq 0$ and $S\subseteq \Sigma^*_\ell$, the \emph{level structure of $S$} is the structure $\str{S}=\AmbStr{S}$. 
\end{definition}
\begin{proposition}
	\label{prop:orders}
	For every $\ell>0$ and $S\subseteq \Sigma^*_\ell$, the level structure $\str{S}=\AmbStr{S}$ satisfies the following properties:
	\begin{enumerate}[label=(P\arabic*)]
		\item\label{P1} $(S,\preceq)$ is a partial order.
		\item\label{P2} $(S,\eltleq)$ is a partial order.
		\item\label{P3} $(S,\lexleq)$ is a linear order.
		\item\label{P4} For every $u,v\in S$ it holds that $u\preceq v\implies u\lexleq v$.\\ ($\lexleq$ is a linear extension of $\preceq$).
		\item\label{P5} For every $u,v\in S$ it holds that $u\eltleq v\implies u\lexleq v$.\\ ($\lexleq$ is a linear extension of $\eltleq$).
		\item\label{P6} For every $u,v,w\in S$ it holds that $u\preceq v\eltleq w\implies u\preceq w$ and $u\eltleq v\preceq w\implies u\preceq w$.
	\end{enumerate}
	Moreover if all words in $S$ are compatible then
	\begin{enumerate}[label=(P\arabic*),resume]
		\item\label{P7} For every $u,v\in S$ it holds that $u\preceq v\implies u\eltleq v$.
	\end{enumerate}
	\begin{proof}
		Properties \ref{P1} and \ref{P4} are Proposition~\ref{prop:pos}. \ref{P2}, \ref{P3}
		\ref{P5} and \ref{P6} follow directly from the definitions. \ref{P7} is Definition~\ref{defn:comp} \ref{defn:comp:p2}.
	\end{proof}
\end{proposition}
\begin{remark}
	\label{rem:approximation}
Level structures can be understood as approximations of a given partial order
with a given linear extension from below (using order $\preceq$) and from above
(using $\eltleq$).  This is a natural analog of the age-set structure in $\ct$.
\end{remark}

\begin{remark}
	One can, perhaps surprisingly, prove that the class $\mathcal K$ of all finite
	structures satisfying properties \ref{P1}, \ref{P2}, \ldots, \ref{P7} is an
	amalgamation class. As a consequence of the construction from Section~\ref{sec:typeP} one gets
	that for each structure $\str A\in\cal{K}$ there exists $\ell>0$ and $S\subseteq \Sigma^*_\ell$
	such that $\str{S}=\AmbStr{S}$ is isomorphic to $\str{A}$.
\end{remark}

\begin{remark}
	This interesting phenomenon of constructing a class of approximations (or, using the terminology of~\cite{zucker2020,Balko2021exact},  the class of all possible aged sets that can appear on some level of the coding tree) of a
	given amalgamation class exists in other cases. For binary free amalgamation classes, this approximation class corresponds to the union $\bigcup_{\rho} P(\rho)$, where the union is taken over all possible \emph{sorts} $\rho$ (see \cite{Balko2021exact} for the definitions). However, the theory of aged coding trees and the sets $P(\rho)$ can be defined for any strong amalgamation class in a finite binary language.
	Note that while for free amalgamation classes the set $P(\rho)$ is always closed under intersections, this need not be the case in general (indeed, it fails for posets).

	Another key difference between the free amalgamation case and partial orders is that for free amalgamation classes, we can enumerate the \fr limit in such a way so that going up and left (that is, by a non-relation) in the coding tree is a safe move, i.e.,\ is an embedding of the level structure from one level to another. Indeed, if this were true for the generic partial order and the coding tree $\ct$ we fixed earlier, one could prove upper bounds for the big Ramsey degrees using forcing arguments much as is done for the free amalgamation case in \cite{zucker2020}.  However, while a weakening of the idea of a ``safe direction" does hold for partial orders (see Proposition~\ref{prop:typeextend}), the proof of Lemma 3.4 from \cite{zucker2020} breaks in the setting of the generic partial order. However, it is possible that the coding tree Milliken theorem still holds. Below, $\mathrm{AEmb}(\ct^\str{A}, \ct)$ refers to the set of \emph{aged embeddings} of the coding tree $\ct^\str{A}$ into $\ct$, the strong similarity maps that respect coding nodes and level structures (see Definition 2.3 of \cite{zucker2020}).
\end{remark}

	\begin{question}
	    Fix a finite partial order $\str{A}$. Let $r< \omega$ and let $\gamma\colon \mathrm{AEmb}(\ct^\str{A}, \ct)\to r$ be a colouring. Is there $h\in \mathrm{AEmb}(\ct, \ct)$ such that $h\circ \mathrm{AEmb}(\ct^\str{A}, \ct)$ is monochromatic? 
	\end{question}

\section{Poset-diaries and level structures}
Given a poset-diary $S$ (Definition~\ref{def:posetdiary}), one can view $\overline{S}$ as a binary branching tree and each level $\overline{S}_\ell$ as a structure $\str{S}_\ell=\AmbStr{\overline{S}_\ell}$
where the structure $\str{S}_{\ell+1}$ is constructed from the structure $\str{S}_{\ell}$ as described in the following proposition.

\begin{proposition}
\label{prop:levelstr}
	Let $S$ be a poset-diary. Then all words in $\overline{S}$ are mutually compatible, and for each  $\ell\leq \sup_{w\in S}|w|$ the structures $\str{S}_\ell=\AmbStr{\overline{S}_\ell}$ and
	$\str{S}_{\ell+1}=\AmbStr{\overline{S}_{\ell+1}}$ are related as follows:
	\begin{enumerate}
		\item If $\overline{S}_\ell$ introduces  a new leaf, then $\str{S}_{\ell+1}$ is isomorphic to $\str{S}_\ell$ with one vertex removed.
		\item If $\overline{S}_\ell$ is splitting, then $\str{S}_{\ell+1}$ 
			is isomorphic to $\str{S}_\ell$ with one vertex $v$ duplicated to $v',v''$ with $v'\lexlt v''$, $v'\not\preceq v''$, $v''\mathbin{\not \preceq} v'$,  and $v'\eltleq v''$.
		\item If $\overline{S}_\ell$ has a  new $\perp$, then $\str{S}_{\ell+1}$ is isomorphic to $\str{S}_\ell$ with one pair removed from   relation $\eltlt$ (and thus one pair added to $\perp$).
		\item If $\overline{S}_\ell$ has a new $\preceq$, then $\str{S}_{\ell+1}$ is isomorphic to $\str{S}_\ell$ extended by one pair in relation $\prec$.
	\end{enumerate}
\end{proposition}
To prove Proposition~\ref{prop:levelstr}, we note the following easy observation.
\begin{observation}
	\label{obs:safeext}
	Let $u\llt v\in \Sigma^*_i$ for some $i\geq 0$ and $c,c'\in \Sigma$ such that $c\lleq c'$ and $(c,c')\neq (\L,\R)$. Then:
	\begin{enumerate}
		\item $u\preceq v\iff u\cont c\preceq v\cont c'$,
		\item $u\perp v\iff u\cont c\perp v\cont c'$,
		\item If $u$ and $v$ are compatible then $u\cont c$ and $v\cont c'$ are compatible.
	\end{enumerate}
\end{observation}
\begin{proof}[Proof of Proposition~\ref{prop:levelstr}]
	Fix  a poset-diary $S$ and level $\ell<\sup_{w\in S}|w|$. We consider individual cases.
	\begin{enumerate}
		\item Leaf vertex $w$: We have that $|\str{S}_\ell|=|\str{S}_{\ell+1}|+1$ since $w$ is the only vertex of $\str{S}_\ell$ not extended to a vertex in $\str{S}_{\ell+1}$.  The desired isomorphism and mutual compatibility follows by Observation~\ref{obs:safeext}.
		\item Splitting of vertex $w$: Here $w$ is the only vertex with two extensions. The desired isomorphism and mutual compatibility follows again by Observation~\ref{obs:safeext}.
		\item New $v\perp w$: Since $v\lexlt w$ are unrelated and thus $v\eltlt w$, we know that $v\cont \R$ and $w\cont \X$ are compatible and that $v\cont \R\perp w\cont \X$ holds. Since we extended by letters $\X$ and $\R$ we know that there are no new pairs in relation $\preceq$.

		Now assume, for contradiction, that there is $u\in \overline{S}_\ell\setminus\{v,w\}$ unrelated to $v$ such that the extension of $u$ in $\overline{S}_{\ell+1}$ is related to $v\cont \R\in \overline{S}_{\ell+1}$. Since all words lexicographically before $v$ are extended by $\X$ and all words lexicographically after $w$ by $\R$, by Observation~\ref{obs:safeext}, we conclude that $v\lexlt u\lexlt w$ and $u$ extends by $\X$. Mutual compatibility follows by analogous argument.
		\item New $v\preceq w$: Since $v\lexlt w$ are unrelated, we know that $v\cont \L$ and
		      $w\cont \R$ are compatible and $v\cont \L\preceq w\cont \R$ holds. To see that no
		      additional pair to relation $\preceq$ was introduced, observe that for $u,u'\in \overline{S}_\ell$,
		      $(u,u')\neq (v,w)$ to be extended to
		      $u\cont \L$, ${u'} \cont \R$
			we have, by assumptions \ref{B1} and \ref{B2},  $u\preceq v$ and $w\preceq u'$. Since $v\lexlt w$ is unrelated we also have $v\eltleq v$. By Proposition~\ref{prop:orders}~\ref{P6} $u\preceq v\lexlt w\implies u\preceq v$ and thus also $u\preceq u'$.

		      It remains to consider the possibility that new pairs are added to relation $\perp$. We again consider individual cases.

		      First consider the case that $u$ is unrelated to $v$ but their extensions are newly in $\perp$. Since $v$ extends by $\L$ we know that $u\lexlt v$ and $v$ extends by $\X$. This contradicts construction of $\overline{S}_{i+1}$.

		      The case that $u$ is unrelated to $w$ but their extensions are newly in $\perp$ follows by symmetry.

			It thus remains to consider the case where $u\lexlt u'$, $u,u'\notin\{v,w\}$, are unrelated in $\overline{S}_i$,  however their extensions are related in $\overline{S}_{i+1}$. It is not possible for $u$ to extend by $\R$ and $u'$ by $\L$. So assume that $u$ extends by $\X$ and $v$ extends by $\L$ (the remaining case follows by symmetry). From this we conclude that $u\lexlt u'\lexlt v$, $u\perp v$ and $u'\not\perp v$. Since $u\not\perp u'$ we again obtain a contradiction with Proposition~\ref{prop:orders} \ref{P4} or \ref{P6}.

		      Mutual compatibility follows by analogous argument.

	\end{enumerate}
\end{proof}

\section{A poset-diary coding $\str{P}$}
\label{sec:typeP}
Recall that $\str{P} = (\omega, \leq_{\str{P}})$ denotes a fixed enumerated generic poset. 
We define a function $\varphi\colon \omega\to \Alphabet^*$ by mapping $j< \omega$ to a word $w$ of length $2j+2$ defined by
putting $(w_{2j},w_{2j+1})=(\L,\R)$ and, for every $i<j$, $(w_{2i},w_{2i+1})$ to $(\L,\L)$ if $j\mathbin{\leq_\str{P}} i$, $(\R,\R)$ if $i\mathbin{\leq_\str{P}} j$, and $(\X,\X)$ otherwise.
We set $T=\overline{\varphi[\omega]}$. The following result is easy to prove by induction.
\begin{proposition}[Proposition 4.7 of \cite{Hubicka2020CS}]
	The function $\varphi$ is an embedding $\str{P}\to (\Sigma^*,\preceq)$.
	More\-over, $\varphi(v)$ is a leaf of $T$ for every $v\in \str{P}$, all words in $T$ are mutually compatible, and if $v,w\in \str{P}$ are incomparable, we have $\varphi(v)\perp\varphi(w)$.
\end{proposition}

We will need the following refinement of this embedding.

\begin{theorem}
	\label{thm:posetemb}
	There exists an embedding $\psi\colon \str{P}\to (\Sigma^*,\preceq)$ such that
	$\psi[\omega]$ is a poset-diary.
\end{theorem}

\begin{proof}
	Fix the embedding $\varphi$ as above and put $T=\overline{\varphi[\omega]}$.
	We proceed by induction on levels of $T$.
	For every $\ell$, we define an integer $N_\ell$ and a function $\psi_\ell\colon T_\ell\to \Sigma^*_{N_\ell}$.
	We will maintain the following conditions:
	\begin{enumerate}
		\item The set $\overline{\psi_\ell[T_\ell]}$ satisfies the conditions of Definition~\ref{def:posetdiary} for all levels with the exception of $N_\ell-1$.
		\item If $\ell>0$, then, for every $u\in T_\ell$, the word $\psi_\ell(u)$ extends $\psi_{\ell-1}(u|_{\ell-1})$.
	\end{enumerate}

	We let $N_0=0$ and put $\psi_0$ to map the empty word to the empty word.
	Now, assume that $N_{\ell-1}$ and $\psi_{\ell-1}$ are already defined.
	We inductively define a sequence of functions $\psi^i_\ell\colon T_\ell\to \Sigma^*_{N_{\ell-1}+i}$. Put $\psi^0_\ell(u)=\psi_{\ell-1}(u|_{\ell-1})$.
	Now, we proceed in steps. At step $j$, apply the first of the following constructions that can be applied and terminate the procedure if none of them applies:
	\begin{enumerate}
		\item\label{posemb1} If $\psi_\ell^{j-1}$ is not injective, let $w\in T_\ell$ be lexicographically least so that $\psi^{j-1}_\ell(w) = \psi^{j-1}_\ell(x)$ for some $x\in T_\ell\setminus\{w\}$. Given $u\in T_\ell$, set $\psi_\ell^j(u) = \psi_\ell^{j-1}(u)^\frown X$ if $u\lexleq w$, and set $\psi_\ell^j(u) = \psi_\ell^{j-1}(u)^\frown R$ if $w\lexlt u$. Then this satisfies the conditions on new splitting at $\psi_{\ell}^{j-1}(w)$ as given in Definition~\ref{def:posetdiary}.
		\item\label{posemb2} If there are words $w$ and $w'$ from $T_\ell$ with $w\lexlt w'$ such that $w\perp w'$ and $\psi^{j-1}_\ell(w)\mathbin{\not \perp}\psi^{j-1}_\ell(w')$ and condition~\ref{A2} of Definition~\ref{def:posetdiary} is satisfied for the value range of~$\psi^{j-1}_\ell$, we construct $\psi^j_\ell$ to satisfy the conditions on new $\perp$ for $\psi^{j-1}_\ell(w)$ and $\psi^{j-1}_\ell(w')$ as given by Definition~\ref{def:posetdiary}.
		\item\label{posemb3}  If there are words $w$ and $w'$ from $T_\ell$ with $w\lexlt w'$ such that $w\prec w'$ and $\psi^{j-1}_\ell(w)\mathbin{\not \prec}\psi^{j-1}_\ell(w')$ and conditions~\ref{B1} and~\ref{B2} of Definition~\ref{def:posetdiary} are satisfied for the value range of~$\psi^{j-1}_\ell$, we construct $\psi^j_\ell$ to satisfy the conditions on new $\prec$ for $\psi^{j-1}_\ell(w)$ and $\psi^{j-1}_\ell(w')$ as given by Definition~\ref{def:posetdiary}.
	\end{enumerate}

	Let $J$ be the largest index for which for which $\psi^J_\ell$ is defined.
	\begin{claim}
	    $\psi^J_\ell$ is an isomorphism  $\AmbStr{T_\ell}\to
	\AmbStr{\psi^J_\ell[T_\ell]}$.
	\end{claim}
	
	\begin{proof}[Proof of claim]
	    Suppose, to the contrary, that this is not true. If $\psi^J_\ell$ is not a bijection, this means that there are $w,w' \in T_\ell$ such that $\psi^J_\ell(w) = \psi^J_\ell(w')$. But then the conditions in~(\ref{posemb1}) are satisfied, a contradiction with maximality of $J$. So $\psi^J_\ell$ is a bijection. Note that the steps of the construction ensure that $\psi^J_\ell$ respects $\lexlt$. We also have $\psi^J_\ell(w)\perp \psi^J_\ell(w')\implies w\perp w'$ and $\psi^J_\ell(w)\preceq \psi^J_\ell(w')\implies w\preceq w'$ for $w,w'\in T_\ell$.

		If there are $w,w' \in T_\ell$ such that $w\lexlt w'$, $w\perp w'$ and $\psi^J_\ell(w)\mathbin{\not \perp}\psi^J_\ell(w')$, pick among all such pairs $(w, w')$ one minimizing $|\{u\in T_\ell: w\lexlt u\lexleq w'\}|$. Proposition~\ref{prop:orders} implies that the conditions in~(\ref{posemb2}) are satisfied, again a contradiction with maximality of $J$.

		So there are  $w,w' \in T_\ell$ such that $w\lexlt w'$, $w\prec w'$ and $\psi^J_\ell(w)\mathbin{\not \prec}\psi^J_\ell(w')$, and we can assume that $w,w'$ maximize $|\{u\in T_\ell: w\lexlt u\lexleq w'\}|$. Proposition~\ref{prop:orders} implies that the conditions in~(\ref{posemb2}) are satisfied, again a contradiction with maximality of $J$. Hence indeed $\psi^J_\ell$ is an isomorphism  $\AmbStr{T_\ell}\to
		\AmbStr{\psi^J_\ell[T_\ell]}$. 
	\end{proof}

	Finally, we put $N_\ell=|\psi^J_\ell(w)|$ for some $w\in T_\ell$ and $\psi_\ell=\psi^J_\ell$.
	Once all the functions $\psi_\ell$ are constructed, we can set $\psi(i) = \psi_{2i+2}(\varphi(i))$. It is easy to verify that this is an embedding $\str{P}\to (\Sigma^*,\preceq)$ such that
	$\psi[\omega]$ is a poset-diary; if not, it fails at some finite level $\ell$, but the construction ensures that every level adheres to the conditions of Definition~\ref{def:posetdiary}.
\end{proof}

\section{Interesting levels and sub-diaries}
We now aim to prove upper bounds for the big Ramsey degrees of $\str{P}$. To do this, we need to define a notion of sub-diary
which corresponds to a subtree of $\Sigma^*$ which preserves all important features of a given subset.  Given $S\subseteq \Sigma^*$ we first
determine which levels contain interesting changes and then define a sub-tree by removing the remaining ``boring" levels from the tree.
This is related to the notion of parameter-space envelopes used in~\cite{Hubicka2020CS}, but more versatile, making it possible to get exact upper bounds.

\begin{definition}[Interesting levels]
	Given $S\subseteq \Sigma^*$ and $i \leq \max\{|s|: s\in S\}$, we call $i$ \emph{interesting for $S$}, or simply \emph{interesting} if $S$ is understood, if
	\begin{enumerate}
		\item the structure $\overline{\str{S}}_i=\AmbStr{\overline{S}_{i}}$ is not isomorphic to  $\overline{\str{S}}_{i+1}=\AmbStr{\overline{S}_{i+1}}$, or 
		\item there exist incompatible $u,v\in \overline{S}_{i+1}$ such that $u|_{i}$ and $v|_{i}$ are compatible, or
        \item 
        there is $u\in S$ with $\lvert u\rvert = i$.
	\end{enumerate}
\end{definition}

\begin{remark}
	Interesting levels are the analog for subsets of $\Sigma^*$ of the notion of \emph{critical level} for a subset of coding nodes in $\ct$; see for instance Definition 5.1 of \cite{zucker2020} or  Definition 5.1.3 of \cite{Balko2021exact} (we note that the two definitions are slightly different).
\end{remark}
Given $S\subseteq \Sigma^*$ and levels $\ell<\ell'$, we call a level $\ell'$ a \emph{duplicate of $\ell$} if $S$ contains no word of length $\ell$ or $\ell'$ and moreover for every $u\in S$ of length greater than $\ell'$ it holds that $u_\ell=u_{\ell'}$.
(Equivalently, all words in $S$ pass the level $\ell'$ in the same way as they pass the earlier level $\ell$.)
By checking definitions of $\lexleq$, $\preceq$ and $\eltleq$ one can derive the following simple result.

\begin{observation}
	\label{obs:duplicatelevels}
	For every $S\subseteq \Sigma^*$ and every $\ell<\ell'$ where level $\ell'$ is a duplicate of level $\ell$, it holds that $\ell'$ is not interesting for $S$.
\end{observation}

\begin{definition}[Embedding types]
	\label{embtype}
Let $I(S)$ be the set of all interesting levels in $S$. Let $\tau_S\colon S\to \Sigma^*$ be the mapping assigning 
to 
each $w\in S$ 
the word created from $w$ by deleting all letters with indices not in $I(S)$.
Define $\tau(S)=\tau_S[S]$, and call this the \emph{embedding type of $S$}.
\end{definition}
	The following observation is a direct consequence of Definition~\ref{def:posetdiary}.

\begin{observation}\label{obs:types}
	For a poset-diary $S$ and $S'\subseteq S$, $\tau(S')$ is a poset-diary.\qed
\end{observation}

We therefore call $\tau(S')$ the \emph{sub-diary} induced by $S'\subseteq S$.

\begin{definition}
\label{Def:Boring}
Recall that for a set $A=\{u^0\lexlt u^1\lexlt\cdots\lexlt u^{n-1}\}\subseteq\Sigma^*_\ell$ (for some $\ell>0$) and word $e\in\Sigma^*_n$, we put $A\cont e=\{{u^i}\cont e_i:0\leq i<n\}$.
We call $e\in \Sigma^*_n$ a \emph{boring extension} of $A$ if $\ell$ is not interesting for $A^\frown e$.
We will denote by $\Pi_A$ the set of all boring extensions of $A$ and by $\Pi_A^*$ the set of all finite words over the alphabet $\Pi_A$. That is, a member of $\Pi_A^*$ is a sequence $w = (w^0,w^1,\ldots,w^{\lvert w\rvert -1})$ such that for every $i$ we have that $w^i\in \Pi_A$. Here we use superscripts instead of subscripts since the ``letters'' of $w$ are themselves words.
\end{definition}

We first prove two properties of boring extensions.

\begin{proposition}
	\label{prop:typeextend}
	For every $\ell\geq 0$, set $S\subseteq \Sigma^*_\ell$ of mutually compatible words, and boring extension $e$ of $S$, there exists a boring extension $e'$ of $\Sigma^*_\ell$ such that $S\cont e\subseteq {\Sigma^*_\ell}\cont e'$.
\end{proposition}
\begin{proof}
	Fix $\ell\geq 0$, $S=\{u^0\lexlt  u^1\lexlt\cdots\lexlt u^{n-1}\}$, $\Sigma^*_\ell=\{v^0\lexlt v^1\lexlt \cdots\lexlt \allowbreak v^{m-1}\}$, and a boring extension $e$ of $S$. For $u\in S$ denote by $i(u)$ the integer $i$ satisfying $u^i=u$.
	For $v\in \Sigma^*_\ell\setminus S$ 
 we say that letter $c\in \Sigma$ is \emph{safe for $v$} if
	for every $0\leq j<n$ such that $u^j$ is compatible with $v$, it holds that ${u^j}\cont e_{j}$ is compatible with $v\cont c$ and $\AmbStr{\{u^j,v\}}$ is isomorphic to $\AmbStr{\{{u^j}\cont e_{j},v\cont c\}}$.

	First we check that for every $v\in \Sigma^*\setminus S$, the set of safe letters for $v$ is non-empty.
	To see this, consider some vertex $v\in \Sigma^*\setminus S$ such that $\X$ is not safe for $v$.  
	Let $u\in S$ witness that $\X$ is not safe for $v$. There are two cases depending on if $u\lexlt v$ or $v\lexlt u$. 
    \begin{enumerate}
        \item 
        If $u\lexlt v$, then we must have that $u\eltleq v$, but $u^\frown e_{i(u)}\not\eltleq v^\frown \X$, implying that $e_{i(u)} = \R$. We will show that $\R$ is safe for $v$.  
        Towards a contradiction, suppose that $w\in S$ witnessed that $\R$ was not safe for $v$. 
    \begin{enumerate}
        \item 
        If $v\lexlt w$, there are two possibilities.
        \begin{itemize}
            \item 
            $v\eltleq w$, but $v^\frown \R\not\eltleq w^\frown e_{i(w)}$. This implies that $e_{i(w)}\neq \R$. However, since $u\eltleq w$, we must have $u^\frown \R\eltleq w^\frown e_{i(w)}$, which implies $e_{i(w)} = \R$, a contradiction.
            \item 
            $\{v^\frown \R, w^\frown e_{i(w)}\}$ are incompatible. If $v\eltleq w$, this reduces to the case above. If $v\not\eltleq w$, then we must have $e_{i(w)} = \L$. But as $\{u^\frown \R, w^\frown \L\}$ are incompatible, this is a contradiction.
        \end{itemize}  
        \item 
        If $w\lexlt v$, the only possibility is that $w\not\prec v$, but $w^\frown e_{i(w)} \prec v^\frown\R$. This implies that $e_{i(w)} = \L$ and that $w\eltleq v$. We cannot have $u\lexlt w$ since $\{u^\frown R,w^\frown L\}$ would not be compatible, and we cannot have $u = w$ since $\R\neq \L$. So $w\lexlt u$; since $e_{i(w)} = \L$ and $e_{i(u)} = \R$, we must have $w\eltleq u$ (otherwise $\{w^\frown\L, u^\frown R\}$ would not be compatible) and $w\prec u$ (since $w^\frown \L\prec u^\frown \R$), contradicting that $u\eltleq v$ and $w\not\prec v$.  
    \end{enumerate}
    \item 
    If $v\lexlt u$, the argument is very similar to the above; one shows that $\L$ is safe for $v$.
    \end{enumerate}

	Now we define $e'_j$ to be $e_{i(v^j)}$ whenever $v^j\in S$; otherwise choose the first letter from $\X,\L,\R$ that is safe for $v^j$ (equivalently from $\X, \R, \L$, since if $\X$ is not safe, the argument above shows that exactly one of $\L$ or $\R$ is safe). 

	To verify that $e'$ is boring, consider some $0\leq i<j\leq m$. First observe that ${\Sigma^*_\ell}\cont e'$ contains no words of length $\ell$. Also, if $(v^i, v^j)$ are incompatible, then any one-letter extensions will yield an isomorphic level structure. So we may assume that $v^i,v^j\notin S$ and that $v^i$, $v^j$ are compatible. We have:
	\begin{enumerate}
		\item $(e'_i,e'_j)\neq (\R,\L)$. Towards a contradiction, assume that $(e'_i, e'_j) = (\R, \L)$, and let $u\in S$ be some vertex that made $\X$ unsafe for $v^i$ and $u'\in S$ be some vertex that made $\X$ unsafe for $v^j$. By the same analysis as above we have $u\lexlt v^i$, $e_{i(u)}=\R$, $u\not\perp v^i$, $v^j\lexlt u'$, $e_{i(u')}=\L$, $v^j\not\perp u'$. From this however we conclude that $u\lexlt u'$, $e_{i(u)}=\R$, $e_{i(u')}=\L$ which contradicts the definition of boring extension and our assumption that $u$ and $u'$ are compatible.
		\item $v^i\not\perp v^j\implies (e'_i, e'_j)\not\in \{(\X,\L), (\R, \X)\}$. Towards a contradiction, assume that $v^i\not\perp v^j$ and $(e'_i, e'_j) = (\X, \L)$, and denote by $u\in S$ some vertex that made $\X$ unsafe for $v^j$. Clearly $v^j\lexlt u$, $v^j\not\perp u$ and $e_{i(u)}=L$. It follows that for every $\ell'<\ell$ we have $v^i_{\ell'}\lleq v^j_{\ell'}\lleq u_{\ell'}$. From this we have that $u$ and $v^i$ are compatible and $u\not\perp v^i$ which makes $\X$ unsafe for $v^i$. A contradiction. A very similar argument shows that $(e'_i, e'_j)\neq (\R, \X)$.
		\item $v^i\not\preceq v^j\implies (e'_i,e'_j)\neq (\L,\R)$. Towards a contradiction, assume that $v^i\not\preceq v^j$ and $(e'_i, e'_j) = (\L, \R)$, and let $u\in S$ be some vertex that made $\X$ unsafe for $v^i$ and $u'\in S$ be the vertex that made $\X$ unsafe for $v^j$.  We have $v^i\lexlt u\lexlt u'\lexlt v^j$ and $e_{i(u)}=\L$, $e_{i(u')}=\R$. It follows that $v^i\not\perp u$, $u\preceq u'$, $u'\not\perp v^j$. Now for every $\ell'<\ell$ we also have $v^i_{\ell'}\llt u_{\ell'}\llt u'_{\ell'}\llt v^j_{\ell'}$, which implies that $v^i\preceq v^j$. A contradiction. \qedhere
	\end{enumerate}
\end{proof}

\begin{proposition}
	\label{prop:typeextend2}
	Let $0\leq \ell\leq \ell'$ be integers, and  let $e$ be a  boring extension  of $\Sigma^*_\ell=\{u^0\lexlt u^1\lexlt \cdots \lexlt u^{n-1}\}$. Put $\Sigma^*_{\ell'}=\{v^0\lexlt v^1\lexlt\cdots\lexlt v^{m-1}\}$ and create the word $e'$ of length $m$ by putting, for every $0\leq i<m$, $e'_i=e_j$ where $j$ satisfies $v^i|_{\ell}=u^j$. Then $e'$ is a boring extension of $\Sigma^*_{\ell'}$.
\end{proposition}
\begin{proof}
    The proof is just a verification of the fact that the relations $\lexleq$, $\preceq$ and $\trianglelefteq$ are determined by the first occurrences of certain combinations of letters which this construction does not change.
    
    If $(v^i, v^j)$ are incompatible, then any one-letter extensions will yield an isomorphic level structure. So we may assume that $(v^i, v^j)$ are compatible.

	Let $v^i,v^j$ be compatible words.  We then check that the structure $\AmbStr{\{v^i,v^j\}}$ is isomorphic to $\AmbStr{\{{v^i}\cont e'_i,{v^j}\cont e'_j\}}$ and that ${v^i}\cont e'_i$ and ${v^j}\cont e'_j$ are compatible. If $e'_i=e'_j$ (in particular, this happens whenever $v^i|_\ell=v^j|_\ell$), the result is clear.
 
    So suppose that $e'_i\neq e'_j$; in particular, this implies that $v^i|_\ell\neq v^j|_\ell$. Note that in this case, the lexicographic order of $v^i$ and $v^j$ is already determined by their restrictions to level $\ell$ (that is, $v^i \lexlt v^j \iff v^i|_\ell\lexlt v^j|_\ell$), hence $(\{v^i|_\ell,v^j|_\ell\},\lexleq)$, $(\{v^i,v^j\},\lexleq)$, and $(\{{v^i}\cont e'_i,{v^j}\cont e'_j\},\lexleq)$ are isomorphic. Consequently, the compatibility of ${v^i}\cont e'_i$ and ${v^j}\cont e'_j$ follows from the compatibility of $v^i$ and $v^j$ and the compatibility of ${v^i|_\ell}\cont e'_i$ and ${v^j|_\ell}\cont e'_j$.

    The fact that $e$ is a boring extension tells us that $\AmbStr{\{v^i|_\ell,v^j|_\ell\}}$ is isomorphic to $\AmbStr{\{{v^i|_\ell}\cont e'_i,{v^j|_\ell}\cont e'_j\}}$. This implies that either $(e'_i,e'_j)\notin \{(L,R),(R,L)\}$ or $\preceq$ is already defined on $\{v^i|_\ell,v^j|_\ell\}$. In either case, $(\{v^i,v^j\},\preceq)$ and $(\{{v^i}\cont e'_i,{v^j}\cont e'_j\},\preceq)$ are isomorphic. A similar argument can be done for $\trianglelefteq$, and $\AmbStr{\{v^i,v^j\}}$ is isomorphic to $\AmbStr{\{{v^i}\cont e'_i,{v^j}\cont e'_j\}}$, that is, $e'$ is a boring extension of $\Sigma^*_{\ell'}$.
\end{proof}

\section{Upper bounds}
We will prove a Ramsey-type theorem for the following kind of embeddings:
\begin{definition}[Shape-preserving functions]
	Given $S\subseteq \Sigma^*$ we call function $f\colon S\to \Sigma^*$ \emph{shape-preserving} if  $\tau_S(w)=\tau_{f[S]}(f(w))$ (see Definition~\ref{embtype})
for every $w\in S$. 
\end{definition}

We will generally consider shape-preserving functions only for those sets $S$ satisfying $S=\tau(S)$ (that is for \emph{embedding types}). However, the next observation follows from the definition without this extra assumption:

\begin{observation}\label{obs:shape}
	Let $f\colon S\to \Sigma^*$ be shape-preserving.
	\begin{enumerate}[label=(\roman*)]
		\item\label{item:composition} For every shape-preserving $h\colon f[S]\to \Sigma^*$ it holds that $h\circ f$ is shape-preserving.
		\item\label{item:length} For all $u,v\in S$, $|u|\leq |v|$ implies that $|f(u)|\leq |f(v)|$.
		\item\label{item:succ} For all $u,v\in S$, $u\sqsubseteq v$ implies that $f(u)\sqsubseteq f(v)$.
		\item\label{item:emb} The function $f$ is an embedding $f\colon \AmbStr{S}\to\AmbStr{\Sigma^*}$ and images of pairs of compatible words are also compatible.
		\item\label{item:boring} If $S$ is a level set (that is, all words in $s$ have same length) and $e$ is a boring extension of $S$, then $e$ is a boring extension of $f[S]$.
	\end{enumerate}
\end{observation}
Given $S\subseteq \Sigma^*$ and $\ell>0$ we denote by $S_{\leq\ell}=\cup_{i\leq \ell} S_i$ the set of all words in $S$ of length at most $\ell$. We also put $S_{<\ell}=S_{\leq\ell-1}$.
\begin{remark}
	If $S=\Sigma^*_{<\ell}$ for some $\ell>0$, then the images of shape-preserving functions are always strong subtrees in the sense of Milliken's tree theorem~\cite{todorcevic2010introduction}. However, the converse is not true:  For example, the function $f\colon \Sigma^*_{<2}\to\Sigma^*$ given by $f(\emptyset) = \emptyset$, $f(\L)= \L\L$, $f(\X)= \X\R$, and $f(\R)= \R\R$ is not shape-preserving because all three levels 0, 1 and 2 of $f[\Sigma^*_{<2}]$ are interesting, while $\im(f)$ is a strong subtree.
\end{remark}
Given $S,S'\subseteq \Sigma^*$ we denote by $\Shape{S}{S'}$ the set of all shape-preserving functions $f\colon S\to \Sigma^*$ such that $f[S]\subseteq S'$.
Given $n< \omega$, we denote by $\nShape{n}{S}{S'}$ the set of all functions in $\Shape{S}{S'}$ which are the identity when restricted to $S_{<n}$. For a shape-preserving function $g\colon S\to\Sigma^*$, Observation~\ref{obs:shape}~\ref{item:length} tells us that the function $\widetilde{g}\colon \{|w|:w\in S\}\to\omega$ defined by $\widetilde{g}(i)=|g(w)|$ for some (equivalently any) $w\in S$ with $|w|=i$ is well defined. If $S$ is finite, denote by $\max(S)$ the ``last'' level $S_k$ where $k=\max_{w\in S}|w|$.

\begin{observation}
	\label{obs:towords}
	Let $S=\overline S=\tau(S)$ be a finite subset of $\Sigma^*$, and write  $\max(S)= S_k = \{u^0\lexlt\allowbreak u^1\lexlt \cdots\allowbreak \lexlt u^n\}$. Then there is a one-to-one correspondence between $\Shape{S}{\Sigma^*}$ and the set $\Shape{S_{<k}}{\Sigma^*}\times \Pi^*_{\max(S)}$ (recall Definition~\ref{Def:Boring}):
	\begin{enumerate}
		\item
			For every $g\in \Shape{S}{\Sigma^*}$ there exists $w\in \Pi^*_{\max(S)}$ such that for every $0\leq i<n$ it holds that  $g(u^i)=g^0(u^i|_{k-1})\cont {u^i_{k-1}}\cont {w^0_i}\cont\cdots\cont w^{|w|-1}_i$.
		\item
		      Conversely also for every $g^0\in \Shape{S_{<k}}{\Sigma}$ and every word $w\in \Pi^*_{\max(S)}$ the function $g'\colon S\to \Sigma^*$ defined by $g'(w)=g^0(w)$ for $|w|<k$ and $$g'(u^i)=g^0(u^i|_{k-1})\cont {u^i_{k-1}}\cont {w^0_i}\cont\cdots\cont w^{|w|-1}_i$$ is shape-preserving.
	\end{enumerate}
\end{observation}
\begin{proof}
	To see the first statement, assume the contrary, and let $g\in \Shape{S}{\Sigma^*}$ be a function for which there is no $w\in \Pi^*_{\max(S)}$ as in $(1)$. Among all such functions $g$ choose one which minimizes $\widetilde{g}(k)$. 

	Because $S=\ol{S}$ we know that $\tilde{g}(k-1)\in I(g[S])$. Because $g$ is shape-preserving it follows that $g(u^i)_{\tilde{g}(k-1)}=u^i_{k-1}$ and $g(u^i)|_{\tilde{g}(k-1)+1}=g(u^i|_{k-1})\cont {u^i_{k-1}}$.
	It follows that $\widetilde{g}(k)\geq \widetilde{g}(k-1)+2$.

	Notice that $\widetilde{g}(k)=\widetilde{g}(k-1)+2$: As $g$ is shape-preserving, $I(g(S))$ contains no levels between $\widetilde{g}(k-1)$ and $\widetilde{g}(k)$ as otherwise we could remove them, getting a counter example $g'$ with smaller $\widetilde{g'}(k)$. If $\widetilde{g}(k) = \widetilde{g}(k-1)+1$, then we can take $w = \emptyset$, and $g$ would not be a counterexample. But now, observe that level $\widetilde{g}(k)+1$ of $g(S)$ is not interesting and thus corresponds to a boring extension in $\Pi^*_{\max(S)}$, which gives a contradiction. 

	The second statement follows by Proposition~\ref{prop:typeextend2}.
\end{proof}

The following pigeonhole lemma is a consequence of Theorem~\ref{thm:CS}.
\begin{lemma}
	\label{lem:pigeonhole}
	Let $S=\overline{S}=\tau(S)$ be a finite non-empty subset of $\Sigma^*$ of mutually compatible words containing at least one non-empty word. Put $k=\max_{w\in S}|w|$. Let $g^0\in \Shape{S_{<k}}{\Sigma^*}$. Denote by $G$ set of all $g\in \Shape{S}{\Sigma^*}$ extending $g^0$ and put $K=\widetilde{g}^0(k-1)$. Then for every finite colouring $\chi\colon G\to \{0,1,\ldots,r-1\}$ there exists $f\in \nShape{K+1}{\Sigma^*}{\Sigma^*}$ such that
	$\chi$ restricted to $\Shape{S}{f[\Sigma^*]}\cap G$ is constant.
\end{lemma}

\begin{proof}
	By Observation~\ref{obs:towords}, the colouring $\chi\colon G\to \{0,1,\ldots, r-1\}$ gives rise to a colouring $\chi'\colon \Pi^*_{\max(S)}\to\{0,1,\ldots,r-1\}$. Apply Theorem~\ref{thm:CS} to obtain $W$ such that $W[\Pi^*_{\max(S)}]$ is monochromatic with respect to $\chi'$. In order to avoid special cases in the upcoming construction, we will assume that $W$ is indexed from 1 and not from 0. 

	For every $u\in \Sigma^*_{\leq K}$ put $f(u)=u$. We will construct the rest of $f$ by induction on levels. Now assume that $f(\Sigma^*_{i-1})$ is already defined for some $i>K$. Put 
	\begin{align*}
	    I = \begin{cases}
	            0 \quad &\text{if } i = K+1,\\
		    \min\{I'< \omega: W_{I'} = \lambda_{i-K-2}\} \quad &\text{if } i > K+1.
	        \end{cases}
	\end{align*}
	Let $J$  be the minimal integer such that $W_{J}=\lambda_{i-K-1}$.
	Now define a sequence of functions $f^{i'}\colon \Sigma^*_i\to\Sigma^*_{K+I+i'}$ for every $I\leq i'< J$. Put $f^{I}(u)=f(u|_{i-1})\cont u_{i-1}$  for every $u\in \Sigma^*_i$.
 Now proceed by induction on $i'$.
	Assume that $f^{i'-1}$ is constructed for some $I<i'<J$ and consider two cases:
	\begin{enumerate}
		\item $W_{i'}=\lambda_j$: Put $f^{i'}(u)=f^{i'-1}(u)\cont u_{j+k+1}$ for every $u\in \Sigma^*_i$. 
		\item $W_{i'}=e$ for some $e\in \Pi_S$:
		      Let $e'$ be the extension of $\Sigma^*_{K+1}$ given by Proposition~\ref{prop:typeextend} 
			for the boring extension $e$ of $g^1(S_k)$, where $g^1$ is defined by putting $g^1(u)\mapsto g^0(u|_{k-1})\cont u_{k-1}$ (by Observation~\ref{obs:shape}~(\ref{item:boring}), boringness of an extension is preserved by shape-preserving functions).
		      Now let $e''$ be the extension given by Proposition~\ref{prop:typeextend2} for the extension $e'$ and level $i$. Enumerate $\Sigma^*_{i}=\{u^0\lexlt u^1\lexlt\cdots\lexlt u^{m-1}\}$ and for $u^j\in S_i$, 
		      put $f^{i'}(u^j)=f^{i'-1}(u)\cont e''_j$ for every $0\leq j<m$.
	\end{enumerate}
	Finally put $f(u)=f^{J-1}(u)$.

	Observe that all levels added by the rules 1 and 2 above are uninteresting since they are either constructed from boring extensions or they are duplicates of levels introduced earlier (in the sense of Observation~\ref{obs:duplicatelevels}).  The last level is interesting because $\tau(\Sigma^*)=\Sigma^*$. From this we get  $f\in \nShape{K+1}{\Sigma^*}{\Sigma^*}$.

 To see that  $\chi$ restricted to $\Shape{S}{f[\Sigma^*]}\cap G$ is constant, pick an arbitrary $g \in \Shape{S}{f[\Sigma^*]}\cap G$. By Observation~\ref{obs:towords} we can decompose $g$ to $g^0$ and a word $w\in \Pi^*_{\max(S)}$ such that $\chi(g)$ is equal to $\chi'(w)$. From our construction it follows that $w\in W(\Sigma^*)$, and so indeed $\chi$ restricted to $\Shape{S}{f[\Sigma^*]}\cap G$ is constant.
\end{proof}

\begin{observation}\label{obs:nobar}
	For every $S=\tau(S)\subseteq \Sigma^*$ and every $f\in \Shape{S}{\Sigma^*}$ there is a unique function $g\in \Shape{\overline{S}}{\Sigma^*}$ extending $f$. It is constructed by putting $g(w|_\ell)=f(w)|_{\widetilde{f}(\ell)}$ for every $w\in S$ and $\ell\leq |w|$. Similarly, for every $g\in \Shape{\overline{S}}{\Sigma^*}$ it holds that $g\restriction S\in \Shape{S}{\Sigma^*}$.
\end{observation}
Notice that $g$ in Observation~\ref{obs:nobar} is well defined by Observation~\ref{obs:shape}~\ref{item:succ}.

\begin{theorem}
	\label{thm:multCS}
	For every finite set $S=\tau(S)\subseteq \Sigma^*$ of mutually compatible words and every finite colouring $\chi\colon \Shape{S}{\Sigma^*}\to\{0,1,\ldots,r-1\}$, there exists
	$f\in \Shape{\Sigma^*}{\Sigma^*}$ such that $\chi$ restricted to $\Shape{S}{f[\Sigma^*]}$ is constant.
\end{theorem}
\begin{proof}
	By Observation~\ref{obs:nobar} we can assume, without loss of generality, that $S=\overline{S}$.
	We will use induction on $k=\max_{w\in S}|w|$.  For $k=0$ we can interpret $\chi$ as a colouring of $\Sigma^*$, apply Theorem~\ref{thm:CS} to obtain a monochromatic infinite-parameter word $W$, define $f(w)=W(w)$ for every $w\in\Sigma^*$, and observe that $f$ is shape-preserving.

	Now fix $S$ such that $k=\max_{w\in S}|w|>0$ and a finite colouring $\chi\colon \Shape{S}{\Sigma^*}\to\{0,1,\ldots,r-1\}$.
	We will make use of the following claim:
	\begin{claim}
		There exists $h\in \Shape{\Sigma^*}{\Sigma^*}$ and a colouring $\chi'\colon \Shape{S_{<k}}{\Sigma^*}\to \{0,1,\ldots,\allowbreak r-1\}$  such that for every finite $g\in \Shape{S}{\Sigma^*}$,  we have $\chi(h\circ g)=\chi'(g \restriction S_{<k})$.
	\end{claim}
	First we show that Theorem~\ref{thm:multCS} follows from the claim. Let $h$ and $\chi'$ be given by the claim. By the induction hypothesis, there exists $f'\in\Shape{\Sigma^*}{\Sigma^*}$ such that $\chi'$ is constant on $f'[\Sigma^*]$. It is easy to check that $f=h\circ f'$ is shape-preserving and that $\chi$ is constant when restricted to $f[\Sigma^*]$.

	\medskip
	It remains to prove the claim. We will obtain $h$ as the limit of the following sequence:
	Pick an enumeration $\Shape{S_{<k}}{\Sigma^*}=\{g^0, g^1,\ldots\}$ such that $0\leq i\leq j$ implies that $\widetilde{g}^i(k-1)\leq\widetilde{g}^j(k-1)$.
	We construct  a sequence of shape-preserving functions $f^0,f^1,\ldots\in \Shape{\Sigma^*}{\Sigma^*}$  such that for every $i>0$ the following is satisfied:
	\begin{enumerate}
		\item $f^i[\Sigma^*]\subseteq f^{i-1}[\Sigma^*]$ and $f^i(u)=f^{i-1}(u)$ for every $u\in \Sigma^*_{<\widetilde{g}^{i-1}(k-1)+1}$.
		\item There exists $c^{i-1}\in \{0,1,\ldots, r-1\}$ such that $\chi(f^i\circ g)=c^{i-1}$ for every $g\in \Shape{S}{\Sigma^*}$ extending $g^{i-1}$.
	\end{enumerate}

	Put $f^0$ to be the identity $\Sigma^*\to\Sigma^*$. Now assume that $f^{i-1}$ is already constructed.  Consider the colouring $\chi^i\colon \Shape{S}{\Sigma^*}\to\{0,1,\ldots,r-1\}$ defined by $\chi^i(g) = \chi(f^{i-1}\circ g)$. Obtain $h^i \in \nShape{\widetilde{g}^{i-1}(k-1)+1}{\Sigma^*}{\Sigma^*}$ by an application of Lemma~\ref{lem:pigeonhole} for the colouring $\chi^i$ and the function $f^{i-1}\circ g^{i-1}$ (as $g^0$ in the statement), and put $f^i=f^{i-1}\circ h^i$.

	\medskip

	Next we construct the limit shape-preserving $h$.
	For every $i\geq 1$ it holds that $f^i(u)=f^{i-1}(u)$ for all $u\in \Sigma^*$ with $|u|\leq \widetilde{g}^{i-1}(k-1)$.
	Because $\widetilde{g}^{i-1}(k-1)$ is an increasing function of $i$ and there is no upper bound on the length of words in $\Sigma^*$, it follows that $h(u) = \lim_{i\to \omega} f^i(u)$ is well-defined for every $u\in \Sigma^*$. Moreover, $h$ is shape-preserving, because the failure of shape-preservation would be witnessed on a finite set. We also put $\chi'(g^i)=c^i$ for every $i\in \omega$.
\end{proof}

Now we are finally ready to prove a big Ramsey result for $\str P$. In the proof of the next result, it is somewhat more natural to work with colourings of \emph{copies} rather than embeddings, where a copy is simply the image of an embedding. Given structures $\bA$ and $\bB$, we denote the set of copies of $\bA$ in $\bB$ by $\binom{\bB}{\bA}$. Note that if $\bX\in \binom{\bB}{\bA}$, then there are exactly $|\aut(\bA)|$-many embeddings $f\in \emb(\bA, \bB)$ with $\im(f) = \bX$. From this, it follows that if one instead defines the big Ramsey degree of $\bA$ in $\bB$ in terms of copies rather than embeddings, the numbers will differ by exactly a factor of $|\aut(\bA)|$ (see for instance Section 4 of \cite{zucker2016topological}). 

\begin{corollary}
\label{cor:big_ramsey}
	For every finite partial order $\str{Q}$, the big Ramsey degree of $\str{Q}$ in the generic partial order $\str{P}$ is at most $|T(\str{Q})|\cdot |\mathrm{Aut}(\str{Q})|$. 
\end{corollary}
\begin{proof}
	Fix a finite partial order $\str{Q}$ and a colouring $\chi$ of $\binom{\bP}{\bQ}$.
	Choose an arbitrary enumeration $T(\str{Q})=\{S_0,S_1,\ldots,S_{n-1}\}$ and an arbitrary embedding $\eta\colon {(\Sigma^*,\preceq)}\to\str{P}$ (which exists since $\str{P}$ is universal).
	Observe that for every $i\in n$, $\chi$ and $\eta$ together induce a colouring $\chi_i$ of $\Shape{S_i}{\Sigma^*}$ by putting $\chi_i(g)=\chi(\im(\eta\circ g)).$ 
	By repeated applications of Theorem~\ref{thm:multCS} we construct a  sequence of functions $\mathrm{Id}=f_0,f_1,\ldots f_{n}\in \Shape{\Sigma^*}{\Sigma^*}$ such that for every $i\in n$ the following is satisfied:
	\begin{enumerate}
		\item $f_{i+1}[\Sigma^*]\subseteq f_{i}[\Sigma^*]$, and
		\item $\chi_i$ restricted to $\Shape{S_i}{f_{i+1}[\Sigma^*]}$ is constant.
	\end{enumerate}
	Let $\psi\colon \str{P}\to{(\Sigma^*,\preceq)}$ be obtained by Theorem~\ref{thm:posetemb}.
	Then $\theta:= \eta\circ f_n\circ \psi\in \emb(\bP, \bP)$, and given $\bX\in \binom{\bP}{\bQ}$, we have that $\chi(\theta[\bX])$ depends only on $\tau(f_n\circ \psi[\bX])$. Hence we have $|\chi[\binom{\theta[\bP]}{\bQ}]| \leq T(\bQ)$.
\end{proof}

\section{The lower bound}\label{sec:lowerBound}

Given a finite partial order $\str{Q}$, a \emph{labeled} poset-diary for $\str{Q}$ is a pair $(S, f)$, where $S\in T(\str{Q})$ and $f\colon \str{Q}\to (S, \prec)$ is an isomorphism. Let $T^{\rm{lab}}(\str{Q})$ denote the set of labeled poset-diaries coding $\str{A}$. Note that $|T^{\rm{lab}}(\str{Q})| = |T(\str{Q})|\cdot |\mathrm{Aut}(\str{Q})|$. Using the embedding $\psi\colon \str{P}\to (\Sigma^*, \preceq)$ constructed in Theorem~\ref{thm:posetemb}, we define a colouring $\chi_\str{Q}\colon \emb(\bQ, \bP)\to T^{\rm{lab}}(\str{Q})$ by setting 
$$\chi_\str{Q}(f)=(\tau(\psi\circ f[Q]), \tau_{\psi\circ f[Q]}\circ\psi\circ f).$$

Given any colouring $\gamma\colon \emb(\bQ, \bP)\to r$ and $h\in \emb(\bP, \bP)$, we obtain a new colouring $\gamma\cdot h\colon \emb(\bQ, \bP)\to r$ by setting $\gamma\cdot h(f) = \gamma(h\circ f)$. In this way, we obtain an action of the monoid $\emb(\bP, \bP)$ on the set of $r$-colourings of $\emb(\bQ, \bP)$ (see Remark~\ref{Rem:Dynamics} for more discussion). We call $\gamma$ a \emph{recurrent} colouring if for any $h\in \emb(\bP, \bP)$, there is $\phi\in \emb(\bP, \bP)$ with $\gamma\cdot (h\circ \phi) = \gamma$. 

The main theorem of this section, Theorem~\ref{thm:embthm}, gives lower bounds for the big Ramsey degrees of $\bP$. 

\begin{theorem}
	\label{thm:embthm}
	For every finite partial order $\str{Q}$, the colouring $\chi_\bQ$ is recurrent. In particular, for any $h\in \emb(\bP, \bP)$, we have
    $$\left\{\chi_\str{Q}(h\circ g): g\in \emb(\bQ, \bP)\right\}=T^{\rm{lab}}(\str{Q}).$$
	Hence the big Ramsey degree of $\bQ$ in $\bP$ is at least $|T^{\rm{lab}}(\bQ)| = |T(\bQ)|\cdot |\aut(\bQ)|$.
\end{theorem}
Combining Theorem~\ref{thm:embthm} with Corollary~\ref{cor:big_ramsey}, we see that the big Ramsey degree of $\bQ$ in $\bP$ is exactly $|T(\bQ)|\cdot |\aut(\bQ)|$, thus proving the first part of Theorem~\ref{thm:main}.

\begin{remark}
\label{Rem:Dynamics}
We briefly discuss the second part of Theorem~\ref{thm:main}, namely the existence of a big Ramsey structure and the connections to topological dynamics. All terminology pertaining to topological dynamics which is not defined here can be found in the introduction of \cite{zucker2017}. As we mentioned above, the monoid $\emb(\bP, \bP)$ acts on the right on the set $r^{\emb(\bQ, \bP)}$. It is natural to equip both $\emb(\bP, \bP)$ and $r^{\emb(\bQ, \bP)}$ with topologies as follows.
\begin{itemize}
    \item 
    We equip $\emb(\bP, \bP)$ with the topology of pointwise convergence, where $h_n\to h$ iff for every $x\in P$, we eventually have $h_n(x) = h(x)$. 
    \item 
    Viewing $r$ as a finite discrete space, we equip $r^{\emb(\bQ, \bP)}$ with the product topology. Thus this space is homeomorphic to Cantor space.
\end{itemize}
With these topologies, the action of $\emb(\bP, \bP)$ on $r^{\emb(\bQ, \bP)}$ is continuous. In particular, restricting our attention to the action of $\aut(\bP)\subseteq \emb(\bP, \bP)$, we have that $r^{\emb(\bQ, \bP)}$ is a (right) $\aut(\bP)$-flow, i.e.\ a continuous (right) action of the topological group $\aut(\bP)$ on a compact Hausdorff space. 

However, from the $\aut(\bP)$ action, we can entirely recover the $\emb(\bP, \bP)$ action as follows. Any topological group $G$ comes equipped with several interesting, compatible uniform structures, among them the \emph{left uniformity}. When $G$ is equipped with its left uniform structure, we can create two larger spaces which canonically embed $G$. One is the \emph{left completion} of $G$, which we denote by $\wh{G}$; the other is the \emph{Samuel compactification} of $G$ with respect to the left uniformity, which we denote by $\Sa(G)$. Up to canonical identifications, we have $G\subseteq \wh{G}\subseteq \sa(G)$. The Samuel compactification is sometimes known as the \emph{greatest $G$-ambit}; it is a $G$-flow, and whenever $X$ is a $G$-flow and $x_0\in X$, there is a (unique) $G$-map $\lambda_{x_0}\colon \sa(G)\to X$ satisfying $\lambda_{x_0}(1_G) = x_0$. This allows us to equip $\sa(G)$ with a semigroup structure; furthermore, this semigroup acts on any $G$-flow $X$ in a way which extends the $G$ action, where given $x\in X$ and $p\in\sa(G)$, we have $x{\cdot}p = \lambda_x(p)$. While the action $X\times \sa(G)\to X$ will not in general be continuous, its restriction to $X\times \wh{G}$ is always continuous. For $G = \aut(\bP)$, we have $\wh{G} = \emb(\bP, \bP)$.

A \emph{completion flow} for $G$ is a $G$-flow $X$ containing a \emph{completion point}, some $x\in X$ such that for any $\eta\in \widehat{G}$, the point $x\cdot\eta\in X$ has dense $G$-orbit, equivalently dense $\wh{G}$-orbit. One of the main theorems from \cite{zucker2017} is that if $G = \aut(\bK)$ for $\bK$ a \fr structure which admits a \emph{big Ramsey structure}, then $G$ admits a completion flow $X$ which is \emph{universal}, i.e.\ which admits a $G$-map onto any other completion flow $Y$. For $G$ of this form, we further have that the universal completion flow is both metrizable and unique up to isomorphism.

There are a few, equivalent ways of defining big Ramsey structures. One way is as follows. Suppose $\bK$ is a \fr structure where each finite $\bA\subseteq \bK$ has big Ramsey degree $r_\bA< \omega$. Then a big Ramsey structure for $\bK$ is a collection of colourings $\{\gamma_\bA: \bA\subseteq\bK \text{ finite}\}$ where:
\begin{enumerate}
    \item 
    For each $\bA\subseteq \bK$ finite, $\gamma_\bA\colon \emb(\bA, \bK)\to r_\bA$ witnesses that the big Ramsey degree of $\bA$ in $\bK$ is at least $r_\bA$. 
    \item 
    The collection $\{\gamma_\bA: \bA\subseteq \bK\text{ finite}\}$ is \emph{coherent}, meaning that whenever $\bA, \bB\subseteq \bK$ are finite and $f\in \emb(\bA, \bB)$, then if $x, y\in \emb(\bB, \bK)$ satisfy $\gamma_\bB(x) = \gamma_\bB(y)$, then also $\gamma_\bA(x\circ f) = \gamma_\bA(y\circ f)$. 
\end{enumerate}
By taking a product of all the various spaces of colourings, we can view $x:= (\gamma_\bA)_\bA$ as a point in this product. Then letting $X = \ol{x\cdot \aut(\bK)}$, one has that $X$ is the universal completion flow of $\aut(\bK)$; it is a completion flow because the point $x$ is a completion point. In particular, for the generic partial order $\bP$, the colorings $(\chi_\bQ)_\bQ$ form a coherent collection of colourings which all witness the exact lower bounds for big Ramsey degrees, and thus we obtain a big Ramsey structure for $\bP$. In fact, we don't need to consider every finite $\bQ\subseteq \bP$; it suffices to consider a representative of each isomorphism type of partial order on a vertex set of size at most $4$ (in a poset-diary, this is enough to encode the relative heights between splitting levels, new $\perp$ levels, and new $\prec$ levels). Thus we obtain a big Ramsey structure for $\str{P}$ in a finite relational language.

This concludes the discussion of Remark~\ref{Rem:Dynamics}.
\end{remark}

The rest of this section proves Theorem~\ref{thm:embthm} and concludes the proof of Theorem~\ref{thm:main}.
First we show, by a repeated application of Lemma~\ref{lem:pigeonhole}, that for every $f\in \emb((\Sigma^*,\preceq), (\Sigma^*,\preceq))$ (not necessarily shape-preserving), there exists 
$g\in \Shape{\Sigma^*}{\Sigma^*}$ such that $f$ preserves some features of the tree structure on $g[\Sigma^*]$.

\begin{lemma}
	\label{lem:canonical}
	For every $f\in \emb((\Sigma^*,\preceq), (\Sigma^*,\preceq))$, there exists $g\in \Shape{\Sigma^*}{\Sigma^*}$ and a sequence $(N_i)_{i\in\omega}$ satisfying:
	\begin{enumerate}
		\item $N_0=0$,
		\item for every $u\in \Sigma^*$ it holds that $N_{|u|}\leq |f(g(u))|< N_{|u|+1}$, and,
		\item for every $u,v\in \Sigma^*$ and $\ell$ such that $u|_\ell = v|_\ell$, it holds that $f(g(u))|_{N_\ell}=f(g(v))|_{N_\ell}$.
	\end{enumerate}
	See Figure~\ref{fig:canonical}.
\end{lemma}
	\begin{figure}[h]
		\includegraphics{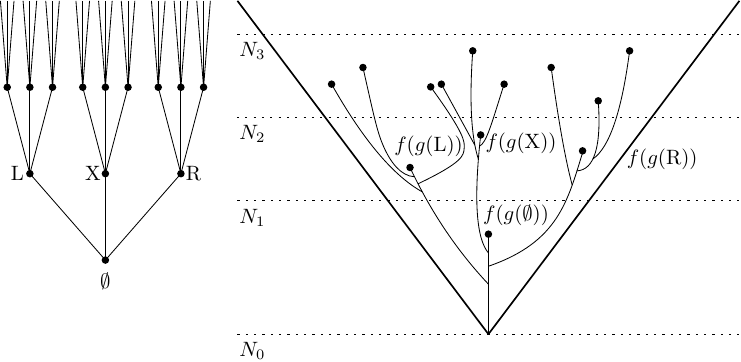}
		\caption{Function $f\circ g$.}
		\label{fig:canonical}
	\end{figure}
\begin{proof}
Fix $f\in \emb((\Sigma^*,\preceq), (\Sigma^*,\preceq))$.
We define a sequence of shape-preserving functions $(g_i)_{i\in\omega}$ and a sequence $(N_i)_{i\in \omega}$ of integers satisfying,
 for every $i>0$, the following three conditions:
\begin{enumerate}[label=(\Alph*)]
	\item\label{g1}$g_i\in \Shape{\Sigma^*}{\Sigma^*}$ and $g_{i-1}\restriction \Sigma^*_{{<}i}=g_i\restriction \Sigma^*_{{<}i}$. 
	\item\label{g2} For every $u\in \Sigma^*_{{<}i}$ it holds that $N_{|u|}\leq f(g_i(u))< N_{|u|+1}$.
	\item\label{g3} For every $u,v\in \Sigma^*$ and $\ell \leq \min(i, |u|,|v|)$ with $u|_\ell=v|_\ell$, we have $f(g_i(u))|_ {N_\ell}=f(g_i(v))|_ {N_\ell}$.
\end{enumerate}

Put $g_0$ to be identity and $N_0=0$. Proceeding by induction, assume that $g_{i-1}$ and $N_{i-1}$ have been constructed for some $i > 0$.  Put $$N_i=\max\left\{|f(g_{i-1}(u))|: u\in \Sigma^*_{i-1}\right\}+1.$$

	Enumerate $\Sigma^*_i=\{w^0,w^1,\ldots,w^{m-1}\}$.  By induction we will construct a sequence of functions $g_{i-1}=g^0_i, g^1_i,\ldots, g^m_i\in\Shape{\Sigma^*}{\Sigma^*}$. Assume that $g^j_i$ is constructed.
	Put $S^j_i=\overline{\{w^j\}}$ and define a colouring $\chi^j_i$ of $\nShape{i}{S^j_i}{\Sigma^*}$ by putting $\chi^j_i(h)=g^j_i(h(w^j))|_{N_i}$. Apply Lemma~\ref{lem:pigeonhole}
	on $\chi^j_i$ and obtain $h^j_i$. Put $g^{j+1}_i=h^j_i\circ g^j_i$.
	Finally, put $g_i=g^m_i$.

	To see that $g_i$ satisfies~\ref{g1}, note that all $h_i^j$'s are shape-preserving functions and that they are the identity when restricted to $\Sigma^*_{<i}$.
	Property~\ref{g2} follows directly from the choice of $N_i$.
	It remains to verify that~\ref{g3} is satisfied. By~\ref{g1} it is enough to verify this for $\ell = i$. Let $u,v\in \Sigma^*$ be such that $u|_i=v|_i$ and let $c^m_i$ be the constant value of $\chi^m_i$ for $m$ satisfying $w^m=u|_i=v|_i$ on $h_i^m$.
	Notice that $|c^m_i|=N_i$ because there are infinitely many images of words extending $w^m$.
	It follows that $g_i(u)|_{N_i}=g_i(v)|_{N_i}=c^m_i$.

It remains to put $g$ to be the limit of sequence $(g_i)_{i\in \omega}$.
\end{proof}

We denote by $d\colon \Sigma^*\to\Sigma^*$ the function that repeats every letter 3 times.  That is, for every $u\in \Sigma^*$ we put $d(u)=u'$ where $|u'|=3|u|$ and for every $\ell<|u|$ it holds that $u'_{3\ell}=u'_{3\ell+1}=u'_{3\ell+2}=u_\ell$.
Note that $d$ is a shape-preserving function.
\begin{lemma}
	\label{lem:ambpreserving}
	Let $f\in \emb((\Sigma^*,\preceq), (\Sigma^*,\preceq))$. Let $g\in \Shape{\Sigma^*}{\Sigma^*}$ and sequence $(N_i)_{i\in \omega}$ be given by Lemma~\ref{lem:canonical}.
	Put $f'=f\circ g\circ d$. 
	Then for every poset-diary $S$ and every $\ell<\sup_{u\in S}|u|$, we have that $\AmbStr{\overline{S}_\ell}\cong \AmbStr{\overline{f'[S]}_{N_{3\ell}}}$.
\end{lemma}
\begin{proof}
	We proceed by induction on $\ell$.  Since $\lvert \overline{S}_0\rvert = 1$ we have: $$\AmbStr{\overline{S}_0}=\AmbStr{\overline{f'[S]}_{N_0}}.$$

	Now assume that $\AmbStr{\overline{S}_\ell}$ is isomorphic to $\AmbStr{\overline{f'[S]}_{N_{3\ell}}}$.  We define the function $\mu$ 
	on $\ol{S}_{\ell+1}$ via $$\mu(u)=f'(u)|_{N_{3\ell+3}}.$$
	We claim that $\mu\colon \AmbStr{\overline{S}_{\ell+1}}\to \AmbStr{\overline{f'[S]}_{N_{3\ell+3}}}$ is an isomorphism:

	\begin{enumerate}
		\item $\mu[\overline{S}_{\ell+1}]\subseteq \overline{f'[S]}_{N_{3\ell+3}}$: Recall that $f'=f\circ g\circ d$. For every $u\in \overline{S}_{\ell+1}$, we have $|d(u)|=3\ell+3$. Let $v\in S$ be a successor of $u$. Notice that $d(v)$ is a successor of $d(u)$.
			Now because $g$ is given by Lemma~\ref{lem:canonical}, we know that both $f'(v)$ and $f'(u)$ are successors of $f'(u)|_{N_{3\ell+3}}$.
		\item For every $u\in S$ with $|u|\geq \ell+1$, we have that $\mu(u|_{\ell+1})\sqsubseteq f'(u)$: This follows directly from the fact that $g$ is constructed using Lemma~\ref{lem:canonical}.
		\item $\mu[\overline{S}_{\ell+1}]\supseteq \overline{f'[S]}_{N_{3\ell+3}}$:  Let $v\in \overline{f'[S]}_{N_{3\ell+3}}$ and 
			choose $u\in S$ such that $v\sqsubseteq f'(u)$.
			It follows from Lemma~\ref{lem:canonical} that $\lvert u\rvert \geq \ell+1$ and $\mu(u|_{\ell+1})=v$.
		\item $\mu$ is injective: 
			Assume that there are $u\lexlt v$ in $\overline{S}_{\ell+1}$ such that $\mu(u)=\mu(v)$.  The induction hypothesis then implies that level $\ell$ is splitting
			and that $u|_\ell = v|_\ell := w$ is the splitting node.  It follows that $u=w\cont \X$ and $v=w\cont \R$.
			Put $w'=d(w)\cont \L$, and observe that $d(u)\perp w'$ and $w'\prec d(v)$.  Since $N_{3\ell}\leq |f\circ g(w')|<N_{3\ell+1}$ and $f\circ g$ is an embedding, we know that
			there is a level $\ell'$ satisfying $N_{3\ell}\leq \ell'<N_{3\ell+1}$ satisfying $f'(u)_{\ell'}=\X$ and $f'(v)_{\ell'}=\R$, which gives a contradiction with $\mu(u)=\mu(v)$, as $\mu(u)\sqsubseteq f'(u)$ and $\mu(v)\sqsubseteq f'(v)$.
		\item $u\lexleq v\implies \mu(u)\lexleq \mu(v)$:
			By the induction hypothesis we only need to consider the case where $u|_\ell=v|_\ell := w$. We must have $u=w\cont \X$ and $v=w\cont \R$.
			Let $u'=w\cont \X\L$ and $v'=w\cont \R\R$.  Since $u'\prec v'$ and thus also $f'(u')\prec f'(v')$, and since $\mu(u)\sqsubseteq f(u')$ and $\mu(v)\sqsubseteq f(v')$, we have $\mu(u)\lexleq \mu(v)$.
            \item $\mu(u)\lexleq \mu(v)\implies u\lexleq v$: This follows from the previous point and the fact that $\lexleq$ is a linear order.
		\item $u\prec v\implies \mu(u)\prec \mu(v)$:  Let $\ell'$ be the minimal level such that $u_{\ell'}=\L$ and $v_{\ell'}=\R$.
			Put $w=d(u|_{\ell'})\cont \L\R$. Observe that $d(u)\prec w\prec d(v)$, that is, $w$ is a \emph{witness} of the fact that $d(u)\prec d(v)$. Observe also that $f'(u)=f(g(d(u)))\preceq f(g(w))\preceq f(g(d(v)))=f'(v)$. Because $\ell'\leq \ell$ we have $|w|\leq 3\ell+2$ and thus $|f(g(w))|< N_{3\ell+3}$. It follows that $\mu(u)\prec \mu(v)$.
		\item $\mu(u)\prec \mu(v)\implies u\prec v$:  Assume that $\mu(u)\prec \mu(v)$ and $u\not\prec v$. Because $\mu(u)\lexlt \mu(v)$, we also have $u\lexlt v$. Consequently, $u\cont \L\X\mathbin{\not \prec} v\cont \X\L$, and so $f'(u\cont \L\X)\not\prec f'(v\cont \X\L)$.  This is a contradiction with $\mu(u)\sqsubseteq f'(u\cont \L\X)$  and $\mu(v)\sqsubseteq f'(v\cont \X\L)$.

		\item $u\eltleq v\implies \mu(u)\eltleq \mu(v)$: Since $u'\prec v'\implies u'\eltlt v'$, we only need to consider the case where $u\eltlt v$ and $u\mathbin{\not \prec} v$, and so $f'(u)\not\prec f'(v)$. We have $f'(u\cont \L)\prec f'(v\cont \R)$, and because $\mu(u)\sqsubseteq f'(u)\sqsubseteq f'(u\cont \L)$ 
			and $\mu(v)\sqsubseteq f'(v)\sqsubseteq f'(v\cont \R)$, it follows that $\mu(u)\eltleq \mu(v)$.
		\item $\mu(u)\eltleq \mu(v)\implies u\eltleq v$: If $u\mathbin{\not \eltleq} v$ then there exists a level $\ell'<\ell+1$ such that $v_{\ell'}\llt u_{\ell'}$. Similarly as in the previous cases, we can produce a witness of this fact and contradict that $\mu(u)\eltleq \mu(v)$. \qedhere
	\end{enumerate}
\end{proof}

\begin{proof}[Proof of Theorem~\ref{thm:embthm}]
	Fix $f\in \emb(\bP, \bP)$. Let $\psi\colon \str{P}\to{(\Sigma^*,\preceq)}$
 be obtained by Theorem~\ref{thm:posetemb}.
	Let $\eta\colon {(\Sigma^*,\preceq)}\to\str{P}$ be an embedding (which exists since $\str{P}$ is universal).
	Now $\psi\circ f\circ\eta$ is an embedding ${(\Sigma^*,\preceq)}\to{(\Sigma^*,\preceq)}$.
	Let $g\colon {(\Sigma^*,\preceq)}\to{(\Sigma^*,\preceq)}$ and $(N_i)_{i\in \omega}$ be obtained by the application of Lemma~\ref{lem:canonical} on $\psi\circ f\circ\eta$.
	Put $f'=\psi\circ f\circ\eta\circ g\circ d$.  We claim that for every poset-diary $S$, we have that $\tau(f'[S])=S$. 
	By Lemma~\ref{lem:ambpreserving}, we know that for every $\ell<\sup_{u\in S}|u|$, we have that $\AmbStr{\overline{S}_\ell}\cong \AmbStr{\overline{f'[S]}_{N_{3\ell}}}$.
	By Proposition~\ref{prop:levelstr}, for every $0<\ell<\sup_{u\in S}|u|$, there is exactly one difference between  $\AmbStr{\overline{S}_{\ell-1}}$ and  $\AmbStr{\overline{S}_{\ell}}$.
	Consequently there is only one interesting level $\ell'$ of $\overline{f'[S]}$ between $N_{3\ell}$ and $N_{3\ell+3}$, and $\AmbStr{\overline{f'[S]}_{N_{\ell'}}}\cong \AmbStr{\overline{S}_\ell}$  while  $\AmbStr{\overline{f'[S]}_{N_{\ell'+1}}}\cong \AmbStr{\overline{S}_{\ell+1}}$. We further note that by the construction of $g$, the map $\tau_{f'[S]}\circ f'$ must be the identity on $S$. We can thus put $\phi= \eta\circ g\circ d$.
\end{proof}

\begin{proof}[Proof of Theorem~\ref{thm:main}]
Given a finite poset $\str{A}$, the fact that the big Ramsey degree of $\str{A}$ is exactly $|T^{\rm{lab}}(\str{A})| = |T(\str{A})|\cdot |\mathrm{Aut}(\str{A})|$ follows from Corollary~\ref{cor:big_ramsey} and Theorem~\ref{thm:embthm}. To conclude that $\str{P}$ admits a big Ramsey structure, we observe that the colourings $\chi_{\str{A}}$ as $\str{A}$ ranges over all finite posets satisfy the hypotheses of Theorem~7.1 from \cite{zucker2017}. Theorem~1.6 from \cite{zucker2017} then shows that $\mathrm{Aut}(\str{P})$ admits a metrizable universal completion flow.
\end{proof}

\section{Concluding remarks}
\subsection{Comparison to big Ramsey degrees of the order of rationals}
It is natural to ask how the characterisation of big Ramsey degrees of partial orders compares to
that of linear orders.
The big Ramsey degrees of the linear order correspond to \emph{Devlin embedding types} which, in our setting, can be defined as follows:
\begin{definition}[Devlin embedding type~\cite{devlin1979}, see also \cite{todorcevic2010introduction}, Definition~6.9]
	\label{def:devlin}
	A set $S\subseteq \{\L,\R\}^*$ is called a \emph{Devlin embedding type}, if no member of $S$ extends any other and precisely one of the following two conditions is satisfied for every level $\ell$ with $0\leq \ell< \sup_{w\in S}|w|$:
	\begin{enumerate}
		\item \textbf{Leaf:}  There is $w\in \overline{S}_\ell$ and
		      \begin{align*}
			      \qquad \overline{S}_{\ell+1} & =(\overline{S}_\ell\setminus \{w\} )\cont \L.
		      \end{align*}
		\item \textbf{Splitting:}  There is $w\in \overline{S}_\ell$ such that
		      \begin{align*}
			      \begin{split}
				      \qquad \overline{S}_{\ell+1}&=\{z\in \overline{S}_\ell:z\lexlt w\}\cont \L\\
				      &\qquad \cup\{w\cont \L,w\cont\R\}\\
				      &\qquad \cup \{z\in \overline{S}_\ell:w\lexlt z\}\cont \L.
			      \end{split}
		      \end{align*}
	\end{enumerate}
	\begin{figure}
		\includegraphics{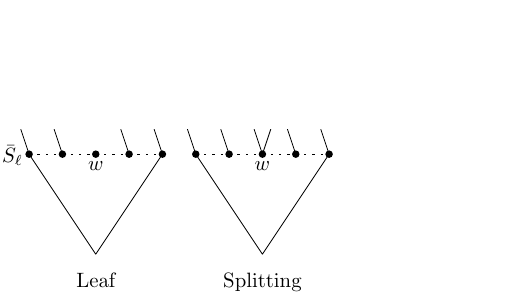}
		\caption{Possible levels in Devlin embedding types.}
		\label{fig:levels2}
	\end{figure}
	See also Figure~\ref{fig:levels2}.
	\begin{remark}
	It is interesting to note the differences between Definition~\ref{def:posetdiary} and Definition~\ref{def:devlin}. Since in linear orders there are
	no incomparable pairs, we only need a two-letter alphabet (there is no need for $X$). Moreover, there are only two interesting events: splitting and
		leaf.  Similarly to Remark~\ref{rem:approximation}, levels of Devlin types can again be seen as approximations of the structure (this time a linear order) represented by the Devlin type.
	However, these approximations are quite easy: they are again linear orders given by the lexicographic ordering where a vertex of the approximation
	corresponds to a linear interval of the final order. We thus do not obtain a new kind of level-structures.
	\end{remark}

	When $S$ is a Devlin embedding type, we call $\overline{S}$ a \emph{Devlin tree}. Given $n\in \omega$, we let $T'(n)$ be the set of all Devlin embedding types of size $n$.
\end{definition}

\begin{theorem}[Devlin~\cite{devlin1979}, see also \cite{todorcevic2010introduction}]
\label{thm:devlin}
	For every $n\in \omega$, the big Ramsey degree of $(n,\leq)$ in the order of rationals equals $|T'(n)|$.
\end{theorem}
Any infinite Devlin tree whose leaves code the rational order is a big Ramsey structure for the rationals~\cite{zucker2017}.

Comparing Devlin trees to poset-diaries is thus relatively natural.  In a Devlin tree, only two types of events happen: splitting and leaf.
In a poset-diary, the splitting event is different.  If $w$ splits on its level then $w\cont \X$ and $w\cont L$ are incomparable in $\preceq$.  This adds a need
for additional two events: new $\prec$ and new $\perp$, deciding the poset structure between the successors of $w$.
\subsection{The triangle-free graph}
We outline how the techniques introduced in this paper can yield a particularly compact characterisation of the big
Ramsey degrees of the generic triangle-free graph. This is a special case of the main result of~\cite{Balko2021exact} (see Example 6.2.10).
However, we now give a more compact description which is similar to Definition~\ref{def:posetdiary}.

We fix an alphabet $\Sigma=\{0,1\}$ and denote by $\Sigma^*$ the set of all
finite words in the alphabet $\Sigma$, 
and by $|w|$ the length of the word $w$ (whose letters are indexed by
natural numbers starting at $0$).  We will denote the empty word by $\emptyset$.

We consider the following triangle-free graph $\str{G}^\triangle$ introduced
in~\cite{Hubicka2020CS}.

\begin{definition}[Graph $\str{G}^\triangle$]~
	\begin{enumerate}
 		\item The vertex set $G$ of $\str{G}^\triangle$ is $\Sigma^*$.
 		\item Given $u,v\in G$ satisfying $|u|<|v|$, we make $u$ and $v$ adjacent if and only if $v_{|u|}=1$ and there is no $i$ satisfying $0\leq i<|u|$ and $u_i=v_i=1$.
		\item There are no other edges in $\str{G}^\triangle$.
	\end{enumerate}
\end{definition}
\begin{definition}[Relation $\perp$]
	Given $u,v\in G$ with $|u|\leq |v|$, we write $u\perp v$ if one of the following is satisfied:
	\begin{enumerate}
		\item There exists $i$ satisfying $0\leq i<|u|$ and $u_i=v_i=1$.
		\item There is no $i$ satisfying $0\leq i<|u|$ and $v_i=1$.
		\item There is no $i$ satisfying $0\leq i<|u|$ and $u_i=1$.
	\end{enumerate}
\end{definition}

Let $\str{R}_3$ denote the generic triangle-free graph. To characterise big Ramsey degrees of $\str{R}_3$, we introduce the following  technical definition.

\begin{definition}[Triangle-free diaries]
	\label{def:trianglediary}
	A set $S\subseteq \Sigma^*$ is called a \emph{triangle-free-type} if no member of $S$ extends any other and precisely one of the following four conditions is satisfied for every $i$ with $0\leq i< \sup_{w\in S}|w|$:
	\begin{enumerate}
		\item \textbf{Leaf:}  There is $w\in \overline{S}_i$ with $w\neq 0^i$ such that for every distinct $u,v\in \{z\in \overline{S}_i\setminus\{w\}:z\not\perp w\}$ it holds that $u \perp v$. 
			Moreover:
			\[
				\overline{S}_{i+1}=\{z\in \overline{S}_i\setminus\{w\}:z\perp w\}\cont 0 \cup \{z\in \overline{S}_i\setminus\{w\}:z\not\perp w\}\cont 1. 
			\]
		\item \textbf{Splitting:}  There is $w\in \overline{S}_i$ such that 
			\[
			      \overline{S}_{i+1}=\overline{S}_i\cont 0\cup \{w\}\cont 1.
			\]
		\item \textbf{Non-splitting first neighbour:}  $0^i\in\overline{S}_i$ and 
			\[
				\overline{S}_{i+1}=(\overline{S}_i\setminus \{0^i\})\cont 0\cup \{{0^i}\cont 1\}.
			\]
		\item \textbf{New $\perp$:} There are distinct words $v,w\in \overline{S}_i$ with $0^i\not\in \{v, w\}$, $v\not\perp w$ such that
			\[
			      \overline{S}_{i+1}=(\overline{S}_i\setminus \{v,w\})\cont 0\cup \{v,w\}\cont 1.
			\]
	\end{enumerate}
	\begin{figure}
		\includegraphics{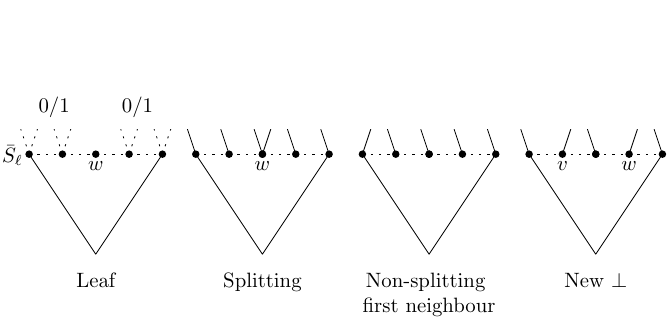}
		\caption{Possible levels in triangle-free diaries.}
		\label{fig:levels3}
	\end{figure}
	See also Figure~\ref{fig:levels3}.
\end{definition}
\begin{remark}
	Tringle-free diares use $0$ to represent non-edges and $1$ to represent edges using the ``passing number'' representation
	which has been known for the Rado graph~\cite{todorcevic2010introduction,Laflamme2006}.
  It is again interesting to understand triangle-free diaries in the context of their level structures.
	While for the Rado graph the level structure is trivial and holds no information except for the lexicographic order,
	in the triangle-free graph it distinguishes two types of vertices.  If vertex has no neighbour defined yet (it is a word $0\cdots 0$), it may split into vertices connected by an edge.
	There is at most one such vertex in a level strcture. On the other hand, if vertex has a neighbour already defined, splitting it to two vertices connected by an edge would
	close a triangle.
	Consequently, it is an interesting event when the first neighbour of a vertex is introduced. Moreover, a given pair of vertices
	may or may not have a common neighbour. If common neighbour exists, there will be no edges between their respective successors.
	This is recorded by relation $\perp$ and the corresponding interesting event.
\end{remark}

Given a triangle-free graph $\str{H}$, we let $T^\triangle(\str{H})$ be the set of all triangle-free-types $S$ such that the structure induced by $\str{G}^\triangle$ on $L(S)$ is isomorphic to $\str{H}$. We can now recover the following result of~\cite{Balko2021exact}. 

\begin{theorem}[\cite{Balko2021exact}]\label{thm:trianglefree}
	For every finite triangle-free graph $\str{H}$, the big Ramsey degree of $\str{H}$ in the generic triangle-free graph $\str{R}_3$ equals $|T^\triangle(\str{H})|\cdot |\mathrm{Aut}(\str{H})|$.
\end{theorem}

\section{Concluding remarks}
The techniques presented in this paper are quite general and can be applied to a relatively wide class of
structures in finite relational languages with relations of arity at most 2,
defined by special triangle conditions isolated in~\cite{balko2021big}.  Important examples
of such structures are metric spaces with finitely many distances, or metric
spaces associated to primitive metrically homogeneous graphs of finite diameter from
Cherlin's catalogue of metrically homogeneous graphs~\cite{Cherlin2013,Aranda2017}.

However, forbidding a configuration on 4 or more vertices (such as in the universal homogeneous $K_4$-free
graph) will run into a limitation of regular trees used here. This can be
overcome in two ways. Either a Ramsey theorem for coding trees can be proved~\cite{dobrinen2017universal,dobrinen2019ramsey,zucker2020,Balko2021exact} or a theorem for trees with finite but unbounded
branching can be used~\cite{Balko2023Sucessor} to represent a more regular tree constructed from a kind of amalgamation of all coding trees corresponding to all possible enumerations of the structure, as used in this paper.

Structures with relations
of arity 3 and more require trees with unbounded branching together with
more involved trees of \emph{weak types}. Upper bounds for structures with very simple restrictions can be obtained using the product form of the Milliken tree theorem~\cite{Hubicka2020uniform,braunfeld2023big}, while more general results can be obtained using trees with a successor operation~\cite{Balko2023Sucessor}, see also the announcement in \cite{typeamalg}.
Even the characterisation of big Ramsey degrees of 3-uniform hypergraphs shows new
and somewhat surprising phenomena, and is a topic of a follow-up paper.

The actual counts of diaries defined in this paper, as well as for diares of $K_4$-free graphs, 
are computed in~\cite{Vodsedalek2025bc,Vodsedalek2025}.

\section*{Acknowledgement}
We would like to thank the two anonymous referees for remarks which significantly improved clarity of this paper, and to Štěpán Vodseďálek for pointing out an important typo in Definitions~\ref{def:posetdiary} and \ref{def:trianglediary}.

D.~Ch., J.~H. and M.~K. were supported by the project 21-10775S of  the  Czech  Science Foundation (GA\v CR). During the revisions D.~Ch., and J.~H. were supported by project 25-15571S of  the  Czech  Science Foundation (GA\v CR). This article is part of a project that has received funding from the European Research Council (ERC) under the European Union's Horizon 2020 research and innovation programme (grant agreement No 810115).
M.~B. was supported by the Center for Foundations of Modern Computer Science (Charles University project UNCE/SCI/004). 
N.~D.\ was supported by National Science Foundation grants DMS-1901753 and DMS-2300896. L.~V. was supported by Beatriu de Pin\'os BP2018, funded by the AGAUR (Government of Catalonia) and by the Horizon 2020 programme No 801370.  A.~Z.\ was supported by NSF grant DMS-2054302 and NSERC grants RGPIN-2023-03269 and DGECR-2023-00412. 

\bibliographystyle{amsalpha}
\bibliography{ramsey.bib}
\end{document}